\documentclass{amsart}
\usepackage{aliascnt}
\usepackage{hyperref}
\usepackage{cleveref}
\usepackage[margin=2.5cm]{geometry}
\numberwithin{equation}{section}
\newtheorem{theorem}{Theorem}[section]
\newaliascnt{lemma}{theorem}
\newtheorem{lemma}[lemma]{Lemma}
\aliascntresetthe{lemma}
\newaliascnt{proposition}{theorem}
\newtheorem{proposition}[proposition]{Proposition}
\aliascntresetthe{proposition}
\newaliascnt{corollary}{theorem}
\newtheorem{corollary}[corollary]{Corollary}
\aliascntresetthe{corollary}
\newaliascnt{definition}{theorem}
\newtheorem{definition}[definition]{Definition}
\aliascntresetthe{definition}
\newaliascnt{example}{theorem}

\aliascntresetthe{example}
\theoremstyle{remark}
\newaliascnt{remark}{theorem}
\newtheorem{remark}[remark]{Remark}
\aliascntresetthe{remark}
\newcommand{\Vol}{\mathrm{Vol}}
\newcommand{\Rm}{\mathrm{Rm}}
\newcommand{\Ric}{\mathrm{Ric}}
\newcommand{\R}{R}
\DeclareMathOperator{\tr}{tr}
\DeclareMathOperator{\diag}{diag}
\DeclareMathOperator{\grad}{grad}
\DeclareMathOperator{\divv}{div}

\begin{document}
\title{A geometric flow associated to $Q$-curvature}

\author{Jiaqi Chen}
\address{Xiamen University of Technology}
\email{chenjiaqi@xmut.edu.cn}

\author{Xudong Kang}
\address{Sun Yat-sen University}
\email{kangxd3@mail2.sysu.edu.cn}

\author{Wei Yuan}
\address{Sun Yat-sen University}
\email{yuanw9@sysu.edu.cn}

\thanks{Xudong Kang and Wei Yuan were partially supported by NSFC (Grant No.~12571065).}

\begin{abstract}
Branson's $Q$-curvature is a natural fourth-order counterpart of scalar curvature in conformal geometry. In analogy with the classical recovery of the Ricci tensor from the linearized scalar curvature, Lin and Yuan introduced a symmetric $2$-tensor $J_g$ through the formal adjoint of the linearized $Q$-curvature and established its explicit decomposition in terms of the Bach tensor and auxiliary curvature terms. In this paper, we study the associated fourth-order geometric flow $\partial_t g = -2J_g$ on closed Riemannian manifolds of dimension $n \ge 4$. We establish short-time existence and uniqueness using a biharmonic DeTurck gauge and derive global and local integral smoothing estimates under curvature bounds. In dimension four, we prove long-time existence assuming a uniform Sobolev inequality and sufficiently small total $L^2$ curvature energy throughout the maximal interval of existence. Under uniform curvature and Sobolev bounds, time integrability of the $L^2$-norm of $J$ implies smooth convergence in fixed coordinates to a $J$-flat metric. Explicit Einstein and product solutions illustrate the evolution and show that small curvature energy alone does not prevent degeneration. Positive initial Yamabe constant and positive total $Q$-curvature give a uniform positive Yamabe lower bound throughout the smooth four-dimensional flow, whereas in higher dimensions total $Q$-curvature can decrease even near flat metrics. Finally, we prove that every closed four-dimensional $J$-soliton is trivial, its metric is Bach-flat with constant $Q$-curvature, and its soliton vector field is Killing.
\end{abstract}

\maketitle

\section{Introduction}\label{sec:introduction}

A central theme in geometric analysis is the construction of canonical metrics by parabolic evolution equations. The most famous example is Hamilton's Ricci flow~\cite{hamilton1982three}
\begin{equation*}
	\frac{\partial g}{\partial t} = -2 \Ric_g,
\end{equation*}
whose stationary points are Ricci-flat metrics and whose normalized variants deform toward Einstein metrics in many settings. Closely related second-order flows include the Yamabe flow~\cite{Hamilton1988,Chow1992,Ye1994}, which evolves within a conformal class toward constant scalar curvature.

For a large class of natural geometric functionals, the associated gradient flows are of order four or higher. Typical examples are the $L^2$-curvature flow of total Riemann curvature tensor~\cite{Streets2008,Streets2012,Streets2013},
\begin{equation*}
	\frac{\partial g}{\partial t} = -\grad \int_M \left| \Rm_g \right|^2 dv_g,
\end{equation*}
the Bach flow on four-manifolds~\cite{BahuaudHelliwell2011,ChenLuQing2023}, and the Calabi flow in K\"ahler geometry~\cite{Calabi1982,ChenHe2008}. We remark that the flow studied here is unrelated to the $J$-flow for K\"ahler potentials. To motivate the J-tensor under consideration, recall the structure of the scalar curvature. Viewing the scalar curvature as a nonlinear map
\begin{equation*}
	R: \mathcal{M} \longrightarrow C^ \infty(M), \qquad g \mapsto R_g
\end{equation*}
on the space $\mathcal{M}$ of Riemannian metrics on a smooth manifold $M$, its linearization at $g$ is the second-order operator
\begin{equation*}
	\gamma_g h := D R_g \cdot h = -\Delta_g ( \tr_g h ) + \delta_g^2 h - \langle \Ric_g,h \rangle_g,
\end{equation*}
where $h \in S_2(M)$ and $\delta_g = -\divv_g$. The formal $L^2$-adjoint $\gamma_g^* : C^\infty(M) \to S_2(M)$ satisfies
\begin{equation*}
	\gamma_g^*f = \nabla_g^2 f - (\Delta_g f) g - f \Ric_g,
\end{equation*}
and therefore
\begin{equation}\label{eq:Ric-from-gamma}
	\Ric_g = -\gamma_g^*1, \qquad R_g = -\tr_g(\gamma_g^*1).
\end{equation}
Thus both the Ricci tensor and the scalar curvature are recovered by evaluating $\gamma_g^*$ on the constant function~$f \equiv 1$.

The natural fourth-order counterpart of scalar curvature in conformal geometry is the $Q$-curvature of Branson~\cite{Branson1985,Branson1995}, closely related to the Paneitz operator~\cite{Paneitz2008} and to the Fefferman--Graham ambient space construction~\cite{FeffermanGraham1985,FeffermanGraham2012}. On an $n$-dimensional Riemannian manifold with $n \ge 3$,
\begin{equation}\label{eq:intro-Q}
	Q_g = A_n \Delta_g R_g + B_n |\Ric_g|_g^2 + C_n R_g^2,
\end{equation}
where the dimensional constants $A_n$, $B_n$, and $C_n$ are given in~\Cref{eq:constants} below. In dimension four one recovers the classical expression
\begin{equation}\label{eq:Q-dim4-intro}
	Q_g = -\frac{1}{6} \Delta_g R_g - \frac{1}{2} |\Ric_g|_g^2 + \frac{1}{6} R_g^2,
\end{equation}
and the Chern--Gauss--Bonnet formula~\cite{ChangGurskyYang2002, ChangYang1995} takes the form
\begin{equation*}
	\int_{M^4} \left(Q_g + \frac{1}{4} |W_g|_g^2 \right) dv_g = 8 \pi^2 \chi(M^4).
\end{equation*}
In particular, on a closed four-manifold the total $Q$-curvature differs from a topological invariant by a multiple of the Weyl energy. The prescribed $Q$-curvature problem and related conformal variational questions have been studied extensively, we refer to~\cite{ChangYang1988,ChangYang1995,Gursky1999,Ndiaye2007} and the surveys~\cite{Chang2005,ChangYangHandbook}.

The same formal procedure that recovers Ricci curvature from the scalar curvature can be applied to $Q$-curvature. Lin and Yuan~\cite{LinYuan2016,LinYuan2017} introduce a symmetric $2$-tensor that stands in the same relation to $Q$-curvature as the Ricci tensor does to scalar curvature. Writing
\begin{equation*}
	\mathcal{Q}: \mathcal{M} \longrightarrow C^ \infty(M),\qquad g \mapsto Q_g,
\end{equation*}
and letting $\Gamma_g = D \mathcal{Q}_g : S_2(M) \to C^\infty(M)$ denote its linearization with formal $L^2$-adjoint $\Gamma_g^*$, the $J$-tensor is defined by
\begin{equation*}
	J_g := - \frac{1}{2} \Gamma_g^*1.
\end{equation*}
Thus $\tr_g J_g = Q_g$ is analogous to~\Cref{eq:Ric-from-gamma}. A metric is called $J$-Einstein if $J_g = \Lambda g$ for some $\Lambda \in C^\infty(M)$, and $J$-flat if $J_g \equiv 0$. In dimension four, $J$-Einstein means precisely Bach-flat, constancy of $Q_g$ is an additional condition. When $\Lambda$ is required to be constant, we say explicitly that the metric is $J$-Einstein with constant coefficient.

The total $Q$-curvature functional
\begin{equation*}
	\mathcal{J} (g) := \int_M Q_g dv_g
\end{equation*}
is closely related to $J_g$, but because variations of $g$ also change the volume form, its $L^2$-gradient on the full space of metrics is not simply a multiple of $J_g$. Precisely,
\begin{equation*}
	D \mathcal{J}_g \cdot h = \int_M \left\langle - 2J_g + \frac{1}{2} Q_g g, h \right\rangle_g dv_g,
\end{equation*}
so $\mathrm{grad} \mathcal{J} = -2 J_g + \frac{1}{2} Q_g g$. Motivated by the Ricci-flow analogy $\Ric = - \gamma_g^*1$ and by the definition $J_g = - \frac{1}{2} \Gamma_g^*1$, we study the following evolution equation
\begin{equation}\label{eq:intro-J-flow}
	\frac{\partial g}{\partial t} = -2 J_g,
\end{equation}
which we call the $J$-flow, which is a quasilinear fourth-order equation for the metric. In dimension four, Lin and Yuan's decomposition yields 
\begin{equation*}
	J_g = \frac{1}{4} Q_g g - \frac{1}{2} B_g,
\end{equation*}
where $B_g$ denotes the Bach tensor, and therefore
\begin{equation}\label{eq:intro-J-flow-dim4}
	\frac{\partial g}{\partial t} = B_g - \frac{1}{2} Q_g g.
\end{equation}
In this sense the four-dimensional $J$-flow is a modification of the Bach flow.

We first state the existence and smoothing results, followed by continuation and conditional convergence results in dimension four. The geometric hypotheses in \Cref{theorem:C,theorem:D} and the additional analytic hypotheses in \Cref{theorem:E,theorem:F} are part of the statements. The role of a uniform positive Yamabe lower bound in the Sobolev estimates is discussed in \Cref{rmk:Yamabe-curvature-scale}; neither initial Yamabe positivity alone nor initial small curvature energy is asserted to propagate the hypotheses of these theorems.

\begin{theorem}\label{theorem:A}
	Let $(M^n, g_0)$ be a closed Riemannian manifold of dimension $n \ge 4$. There exist $T = T(g_0) > 0$ and a unique smooth solution $g(t)$ of the $J$-flow on $M^n \times [0,T]$ with $g(0) = g_0$.
\end{theorem}

The proof proceeds by a biharmonic DeTurck gauge, which restores strict parabolicity. This is in the same spirit as the gauge-fixing arguments for the Ricci flow~\cite{DeTurck1983,hamilton1982three} and for fourth-order curvature flows~\cite{Streets2008}.

\begin{theorem}\label{theorem:B}
	Let $(M^n, g_0)$ be a closed Riemannian manifold of dimension $n \ge 4$, and let $g(t)$ be a smooth solution of the $J$-flow on $M^n \times [0, \tau]$ with $g(0) = g_0$. Suppose that $|\Rm_{g(t)}| \le K$ and $\tau \le K^{-2}$. Then for every $m \ge 0$ there exists $C_m = C_m(n, m)$ such that
	\begin{equation}\label{eq:BBS}
		\left\| \nabla^m \Rm \right\|_{L^2(g(t))} \le C_m \left\| \Rm \right\|_{L^2(g(0))} t^{-\frac{m}{4}}
	\end{equation}
	for all $0 < t \le \min\{\tau, 1\}$.
\end{theorem}

\begin{theorem}\label{theorem:C}
	Let $(M^n, g_0)$ be a closed Riemannian manifold of dimension $n \ge 4$. Fix $p\in M$ and $r > 0$, and let $g(t)$ be a smooth solution of the $J$-flow on $B_{2r}(p,g_0) \times [0,\tau]$ with $g(0) = g_0$. Assume $|\Rm| \le K$, $\tau \le K^{-2}$. Suppose that there are time-independent smooth cutoffs $0 \le \phi,\Phi \le 1$ supported in $B_{2r}(p,g_0)$, with $\phi = 1$ on $B_r(p,g_0)$ and $\Phi = 1$ on a neighborhood of $\overline{\{ \phi > 0\}}$, satisfying
	\begin{equation}\label{eq:cutoff-estimate}
		\sup_{0 \le t \le \tau} \left(|\nabla \phi| + |\nabla^2 \phi| + |\nabla \Phi| + |\nabla^2 \Phi| \right)_{g(t)} \le L.
	\end{equation}
	Then for every $m \ge 0$ there is $C = C(n,m,K,L)$ such that
	\begin{equation}\label{eq:local-BBS}
		\int_{B_r(p,g_0)} |\nabla^m \Rm|_{g(t)}^2 dv_{g(t)} \le C t^{-\frac{m}{2}} \sup_{0 \le s \le t} \int_{B_{2r} (p,g_0)} |\Rm|_{g(s)}^2 dv_{g(s)}
	\end{equation}
	for $0 < t \le \min\{\tau, 1\}$. 
\end{theorem}

\begin{theorem}\label{theorem:D}
	Let $(M^4, g_0)$ be a closed Riemannian four-manifold, and let $g(t)$ be the maximal smooth solution of the $J$-flow on $M^4 \times [0, T_{\max})$ with $g(0) = g_0$. Suppose there are $C_S < \infty$ and $r_0 > 0$, independent of time, such that for every smooth function $u$ and every $t < T_{\max}$,
	\begin{equation}\label{eq:uniform-Sobolev-D}
		\|u\|_{L^4(g(t))}^2 \le C_S  \left( \|\nabla u\|_{L^2(g(t))}^2 + r_0^{-2} \|u\|_{L^2(g(t))}^2 \right).
	\end{equation}
	There is $\varepsilon(C_S) > 0$, independent of $g_0$ and $r_0$, such that
	\begin{equation}\label{eq:global-L2-small}
		\sup_{0 \le t < T_{\max}} \int_M |\Rm_{g(t)}|^2 dv_{g(t)} < \varepsilon(C_S)
	\end{equation}
	implies $T_{\max} = \infty$. 
\end{theorem}

The next theorem gives a criterion for smooth convergence without assuming compactness of the family of metrics. Its proof uses estimates for the $J$-tensor to deduce convergence in fixed coordinates from time integrability of the $L^2$-norm of $J$.

\begin{theorem}\label{theorem:E}
	Let $(M^4, g_0)$ be a closed Riemannian four-manifold, and let $g(t)$ be a smooth solution of the $J$-flow on $M^4 \times [0, \infty)$ with $g(0) = g_0$. Suppose that $|\Rm_{g(t)}| \le K$ for all $t \ge 0$ and that the Sobolev inequality in \Cref{eq:uniform-Sobolev-D} holds for all $t \ge 0$ with time-independent constants $C_S < \infty$ and $r_0 > 0$. If
	\begin{equation}\label{eq:J-time-integrability}
		\int_0^\infty \|J_{g(t)}\|_{L^2(g(t))} dt < \infty,
	\end{equation}
	Then $g(t)$ converges as $t \to \infty$ in every $C^k(\bar{g})$ norm, for any fixed smooth background metric $\bar{g}$, to a $J$-flat Riemannian metric.
\end{theorem}

Combining this criterion with continuation and the variational identities gives the following conditional convergence result.

\begin{theorem}\label{theorem:F}
	Let $(M^4, g_0)$ be a closed Riemannian four-manifold, and let $g(t)$ be the maximal smooth solution of the $J$-flow on $M^4 \times [0, T_{\max})$ with $g(0) = g_0$. Suppose that the hypotheses of \Cref{theorem:D} hold, so that $T_{\max} = \infty$. Assume that
	\begin{equation}\label{eq:Q-time-integrability-F}
		\int_0^{\infty} \|Q_{g(t)}\|_{L^2(g(t))} dt < \infty,
	\end{equation}
	and that there exist $C_{\mathrm L} < \infty$ and $0 < \theta \le \frac{1}{2}$ such that, for all sufficiently large $t$,
	\begin{equation}\label{eq:LS-hypothesis-F}
		\left[ -\mathcal J(g(t)) \right]^{1-\theta} \le C_{\mathrm L} \|B_{g(t)}\|_{L^2(g(t))}.
	\end{equation}
	Then $g(t)$ converges as $t \to \infty$ in every $C^k(\bar{g})$ norm, for any fixed smooth background metric $\bar{g}$, to a $J$-flat Riemannian metric.
\end{theorem}

The remainder of the paper is organized as follows. \Cref{sec:notation} fixes notation and collects the basic formulas for $Q$-curvature, the operators $\Gamma_g$ and $\Gamma_g^*$, and the $J$-tensor. \Cref{sec:examples} studies explicit Einstein and product solutions, compares the normalized $J$-flow with the Ricci flow, and illustrates the role of the convergence hypotheses. \Cref{sec:short-time} establishes short-time existence. \Cref{sec:integral-smoothing} derives curvature evolution equations, integral smoothing estimates and the tensor Sobolev estimate. \Cref{sec:long-time} proves continuation and long-time existence under Sobolev control. \Cref{sec:J-estimates} establishes estimates for the $J$-tensor and proves smooth convergence under time integrability of its $L^2$-norm. \Cref{sec:variational} develops the variational identities, proves the Yamabe lower bound and the higher-dimensional decrease of total $Q$-curvature, and combines four-dimensional monotonicity with the preceding estimates to establish conditional convergence. Finally, \Cref{sec:J-soliton-future} proves the triviality of closed four-dimensional $J$-solitons and discusses variational questions, conformal geometry and the stability of flat tori.

\section{Preliminaries}\label{sec:notation}

Throughout this paper, $(M^n, g)$ denotes a smooth closed Riemannian manifold. The $J$-tensor is defined for $n \ge 4$, the $J$-flow and all flow results below assume $n \ge 4$. We write $\mathcal{M}$ for the space of smooth Riemannian metrics on $M$ and $S_2(M)$ for the space of smooth symmetric $(0, 2)$-tensors. The Levi-Civita connection of $g$ is denoted by $\nabla = \nabla_g$, and $d v_g$ is the Riemannian volume form. 

\subsection{Curvature conventions}\label{subsec:curvature-conventions}

Our sign conventions for curvature follow \cite{ChowLuNi2006}. The Riemann curvature tensor is determined by
\begin{equation*}
    R(X,Y) Z = \nabla_X \nabla_Y Z - \nabla_Y \nabla_X Z - \nabla_{[X,Y]}Z,
\end{equation*}
and in local coordinates
\begin{equation*}
    R_{ijkl} = g \left( R(\partial_i, \partial_j) \partial_k, \partial_l \right).
\end{equation*}
The Ricci curvature tensor and the scalar curvature are given by
\begin{equation*}
    \R_{jk} = g^{il} R_{ijkl}, \qquad R = g^{jk} \R_{jk}.
\end{equation*}
We write
\begin{equation*}
	E = \Ric - \frac{R}{n}g
\end{equation*}
for the traceless Ricci tensor and $W$ for the Weyl curvature. For $n \ge 4$ one has the orthogonal decomposition
\begin{equation}\label{eq:Rm-decomp}
	\Rm = W + \frac{1}{n-2} E \odot g + \frac{R}{2n(n-1)} g \odot g,
\end{equation}
where $\odot$ denotes the Kulkarni--Nomizu product
\begin{equation}\label{eq:KN}
	(\alpha \odot \beta)_{ijkl} = \alpha_{il} \beta_{jk} - \alpha_{ik} \beta_{jl} + \alpha_{jk} \beta_{il} - \alpha_{jl} \beta_{ik}.
\end{equation}
In particular,
\begin{equation}\label{eq:Rm-norm}
	| \Rm|^2 = |W|^2 + \frac{4}{n-2} |E|^2 + \frac{2}{n(n-1)} R^2.
\end{equation}

The divergence of a symmetric $2$-tensor $h$ is $(\delta_g h)_i = - \nabla^j h_{ji}$, and we write $\delta_g^2 h = \delta_g (\delta_g h)$. The rough Laplacian on tensors is $\Delta_g = g^{ij} \nabla_i \nabla_j$, so that $\Delta_g$ has non-positive spectrum on functions. We adopt the following definition of the Lichnerowicz Laplacian
\begin{equation*}
	(\Delta_L h)_{ij} = \Delta_g h_{ij} - 2R_{ikjl} h^{kl} - \R_i{}^k h_{kj} - \R_j{}^k h_{ki}.
\end{equation*}
For symmetric $2$-tensors $\alpha$ and $\beta$ we use the product
\begin{equation*}
	(\alpha \times \beta)_{ij} := \alpha_i{}^k \beta_{kj},
\end{equation*}
following~\cite{LinYuan2016,LinYuan2017}, and we write $\mathring{\Rm} \cdot \Ric$ for the contraction $(\mathring{\Rm} \cdot \Ric)_{ij} = R_{kijl} \R^{kl}$. With these slots,
\begin{equation*}
	\tr_g(\mathring{\Rm} \cdot \Ric) = |\Ric|^2, \qquad \mathring{\Rm} \cdot g = \Ric.
\end{equation*}
These identities fix the signs in the explicit formulas below.

\subsection{Q-curvature and the Schouten tensor}\label{subsec:Q-Schouten}

The dimensional constants in~\Cref{eq:intro-Q} are
\begin{equation}\label{eq:constants}
	A_n = -\frac{1}{2(n-1)}, \qquad B_n = -\frac{2}{(n-2)^2}, \qquad C_n = \frac{n^2 (n-4) + 16 (n-1)}{8 (n-1)^2 (n-2)^2}.
\end{equation}
The Schouten tensor
\begin{equation*}
	S_g := \frac{1}{n-2} \left(\Ric_g - \frac{R_g}{2 (n-1)} g \right)
\end{equation*}
satisfies
\begin{equation*}
	\tr_g S_g = \frac{R_g}{2(n-1)}, \qquad \Ric_g = (n - 2) S_g + (\tr_g S_g) g.
\end{equation*}
In terms of $S_g$, the $Q$-curvature admits the equivalent expression~\cite{Branson1995,FeffermanGraham2012}
\begin{equation*}
	Q_g = -\Delta_g(\tr_g S_g) - 2|S_g|_g^2 + \frac{n}{2} (\tr_g S_g)^2.
\end{equation*}

\subsection{Linearization of Q-curvature}\label{subsec:Gamma}

We first recall the standard linearizations of the Ricci tensor and the scalar curvature, which motivate the fourth-order formulas below. For $h\in S_2(M)$,
\begin{equation}\label{eq:Ric-linearization}
	(D \Ric_g \cdot h)_{ij} = -\frac{1}{2} \left( \Delta_g h_{ij} + \nabla_i(\delta_g h)_j + \nabla_j(\delta_g h)_i + \nabla_i \nabla_j \tr_g h \right) + \text{l.o.t.},
\end{equation}
where the lower-order terms involve curvature contracted with $h$. Similarly,
\begin{equation}\label{eq:R-linearization}
	DR_g \cdot h = -\Delta_g \tr_g h + \delta_g^2 h - \Ric_g \cdot h .
\end{equation}

Following Lin--Yuan~\cite{LinYuan2016,LinYuan2017}, the linearization $\Gamma_g = D\mathcal{Q}_g$ of the $Q$-curvature map at a metric $g$ is the fourth-order operator
\begin{equation*}
	\begin{split}
		\Gamma_g h
		=& A_n \left( -\Delta_g^2 (\tr_g h) + \Delta_g \delta_g^2 h - \Delta_g (\Ric_g \cdot h) + \frac{1}{2} \mathrm{d} R_g \cdot \left(\mathrm{d} (\tr_g h) + 2 \delta_g h \right) - \nabla_g^2 R_g \cdot h \right) \\
		& - B_n \left( \Ric_g \cdot \Delta_L h + \Ric_g \cdot \nabla_g^2(\tr_g h) + 2 \Ric_g \cdot \nabla_g (\delta_g h) + 2(\Ric_g \times \Ric_g) \cdot h \right)\\
		& + 2 C_n R_g \left( -\Delta_g(\tr_g h) + \delta_g^2 h - \Ric_g \cdot h \right).
	\end{split}
\end{equation*}
Its formal $L^2$-adjoint $\Gamma_g^*: C^\infty(M) \to S_2(M)$ is given by
\begin{equation*}
	\begin{split}
		\Gamma_g^* f
		=& A_n \left( -g \Delta_g^2 f + \nabla_g^2 \Delta_g f - \Ric_g \Delta_g f + \frac{1}{2} g \delta_g(f \mathrm{d} R_g) + \nabla_g(f \mathrm{d} R_g) - f \nabla_g^2 R_g \right) \\
		& - B_n \left( \Delta_g(f \Ric_g) + 2 f \mathring{\Rm}_g \cdot \Ric_g + g \delta_g^2 (f \Ric_g) + 2 \nabla_g \delta_g (f \Ric_g) \right)\\
		& - 2C_n \left( g \Delta_g (f R_g) - \nabla_g^2 (f R_g) + f R_g \Ric_g \right).
	\end{split}
\end{equation*}
In particular, setting $f \equiv 1$ and using the contracted Bianchi identity $\delta_g \Ric_g = -\frac{1}{2} \mathrm{d} R_g$ yields
\begin{equation}\label{eq:Gamma-star-1}
	\Gamma_g^* 1 = -B_n \Delta_g \Ric_g - \left(\frac{1}{2} A_n + \frac{1}{2} B_n + 2C_n \right) g \Delta_g R_g + \left(B_n + 2 C_n \right) \nabla_g^2 R_g - 2B_n \left(\mathring{\Rm}_g \cdot \Ric_g \right) - 2C_n R_g \Ric_g.
\end{equation}

\subsection{The Bach tensor and an auxiliary tensor}\label{subsec:Bach-T}

To express $J_g$ geometrically we introduce two auxiliary tensors built from the Schouten tensor. The Bach tensor is
\begin{equation*}
	B_g := \Delta_L S_g - \nabla_g^2(\tr_g S_g) + n \left(S_g \times S_g - \frac{1}{n} |S_g|_g^2 g \right).
\end{equation*}
In dimension four this coincides with the classical Bach tensor
\begin{equation}\label{eq:Bach-classical}
	B_{ij} = \nabla^k \nabla^l W_{iklj} + \frac{1}{2} \R^{kl} W_{iklj}
\end{equation}
of~\cite{Bach1921,Besse1987}, which is symmetric, traceless, and divergence-free. With the full tensor norm used here, the Weyl functional satisfies
\begin{equation*}
	\mathcal{W}(g) = \int_M |W|^2 dv_g, \qquad \mathrm{grad} \mathcal{W} = -4 B.
\end{equation*}
We also set
\begin{equation*}
	T_g := (n - 2) \left( \nabla_g^2 (\tr_g S_g) - \frac{1}{n} g \Delta_g ( \tr_g S_g) \right) + 4(n - 1) \left(S_g \times S_g - \frac{1}{n} |S_g|_g^2 g \right) - n^2(\tr_g S_g) \mathring{S}_g,
\end{equation*}
where $\mathring{S}_g = S_g - \frac{1}{n}(\tr_g S_g) g$.

\subsection{The J-tensor}\label{subsec:J-tensor}

We now collect the basic identities for the $J$-tensor that will be used throughout the paper. 

\begin{definition}\label{def:J-tensor}
	Let $(M^n,g)$ be a Riemannian manifold with $n \ge 4$. The $J$-tensor is the symmetric $2$-tensor
	\begin{equation*}
		J_g := - \frac{1}{2} \Gamma_g^*1.
	\end{equation*}
\end{definition}

The following decomposition is due to Lin and Yuan~\cite[Theorem~1.3]{LinYuan2017}.

\begin{proposition}\label{prop:J-explicit}
	Let $(M^n, g)$ be a Riemannian manifold with $n \ge 4$. Then
	\begin{equation*}
		J_g = -\frac{1}{2} \Gamma_g^*1 = \frac{1}{n} Q_g g - \frac{1}{n-2} B_g - \frac{n-4}{4 (n-1) (n-2)} T_g.
	\end{equation*}
	In particular, when $n = 4$,
	\begin{equation}\label{eq:J-dim4}
		J_g = \frac{1}{4} Q_g g - \frac{1}{2} B_g,
	\end{equation}
	and the $J$-flow becomes
	\begin{equation}\label{eq:J-flow-dim4-again}
		\frac{\partial g}{\partial t} = -2 J_g = B_g - \frac{1}{2} Q_g g.
	\end{equation}
\end{proposition}

\begin{proposition}\label{prop:trace-J}
	For any Riemannian manifold $(M^n, g)$ with $n \ge 4$,
	\begin{equation*}
		\tr_g J_g = Q_g.
	\end{equation*}
	Consequently, along the $J$-flow, the volume element evolves by
	\begin{equation}\label{eq:volume-evolution}
		\frac{\partial}{\partial t} dv_g = -Q_g dv_g.
	\end{equation}
\end{proposition}
\begin{proof}
	Both $B_g$ and $T_g$ are trace-free, so tracing \Cref{prop:J-explicit} gives $\tr_g J_g = Q_g$. The volume formula follows from $\partial_t dv_g = \frac{1}{2} \tr_g(\partial_t g) dv_g$ and $\partial_t g = -2 J_g$.
\end{proof}

\begin{proposition}\label{prop:J-Schur}
	The tensor $J_g$ satisfies $\divv_g J_g = \frac{1}{4} dQ_g$. Consequently a $J$-Einstein metric has constant $Q$ on each connected component when $n \ne 4$. In dimension four,
	\begin{equation*}
		J_g \text{-Einstein} \quad \Longleftrightarrow \quad B_g = 0,
	\end{equation*}
	and $J$-Einstein with constant coefficient is equivalent to $B_g = 0$ and $Q_g$ constant.
\end{proposition}
\begin{proof}
	The functional $\int_M Q_g dv_g$ is invariant under pullback by diffeomorphisms, and its first variation is the $L^2$-pairing with $-2 J + \frac{1}{2} Qg$. Pairing this tensor with a Lie derivative and integrating by parts shows that
	\begin{equation*}
		\divv(-2 J + \frac{1}{2} Q g) = 0.
	\end{equation*}
	For $J$-Einstein, $J = \frac{1}{n} Q g$, this becomes $(\frac{1}{n} - \frac{1}{4}) dQ = 0$. At $n = 4$ the equation is an identity, and both assertions follow instead from \Cref{eq:J-dim4}. 
\end{proof}

Expanding~\Cref{eq:Gamma-star-1} and collecting terms by differential order yields an explicit curvature formula for $J_g$.

\begin{corollary}\label{cor:J-full}
	The $J$-tensor admits the expression
	\begin{equation}\label{eq:J-full}
		\begin{split}
			J_g 
			& = -\frac{(3n-4)(n-4)}{8(n-1)^2(n-2)^2} g \Delta_g R_g - \frac{n^2-10n+12}{8(n-1)^2(n-2)} \nabla_g^2 R_g - \frac{1}{(n-2)^2} \Delta_g \Ric_g \\
			& \quad - \frac{2}{(n-2)^2} \mathring{\Rm}_g \cdot \Ric_g + \frac{n^2(n-4)+16(n-1)}{8(n-1)^2(n-2)^2} R_g \Ric_g.
		\end{split}
	\end{equation}
\end{corollary}

\begin{corollary}\label{cor:J-dim4-full}
	When $n = 4$, formula~\Cref{eq:J-full} specializes to
	\begin{equation}\label{eq:J-dim4-full}
		J_g = \frac{1}{12} \nabla_g^2 R_g - \frac{1}{4} \Delta_g \Ric_g - \frac{1}{2} \mathring{\Rm}_g \cdot \Ric_g + \frac{1}{6} R_g \Ric_g.
	\end{equation}
\end{corollary}

\section{Einstein and product solutions}\label{sec:examples}

Before developing the general theory, we study explicit solutions on Einstein four-manifolds and products of constant-curvature surfaces. These examples illustrate four features of the flow: homothetic shrinking of every non-Ricci-flat Einstein metric; convergence after rescaling to a $J$-Einstein metric on $S^2 \times \Sigma_\gamma$, which admits no Einstein metric; stability of the symmetric metric on $S^2 \times S^2$ under the volume-normalized flow within the product family; and failure of convergence on $S^2 \times \mathbb T^2$ despite decreasing curvature energy and positive Yamabe constant at every finite time. The comparison with the Ricci flow on $S^2 \times \Sigma_\gamma$ parallels the behavior observed by Streets~\cite{Streets2008} for the gradient flow of $\int_M |\Rm|^2$.

\subsection{\texorpdfstring{$J$-flow on Einstein four-manifolds}{J-flow on Einstein four-manifolds}}\label{subsec:Einstein-solutions}

Let $(M^4, g_0)$ be an Einstein four-manifold, $\Ric_0 = \frac{1}{4} R_0 g_0$ with $R_0$ constant. Then it is Bach flat and
\begin{equation*}
	Q_0 = -\frac{1}{2} |\Ric_0|_{g_0}^2 + \frac{1}{6} R_0^2 = \frac{1}{24} R_0^2 \ge 0,
\end{equation*}
with equality if and only if $g_0$ is Ricci-flat. Hence $J_0 = \frac{1}{4} Q_0 g_0$, so $g_0$ is $J$-Einstein, and the flow preserves the homothety class. Writing $g(t) = f(t) g_0$ and using the scaling relations $Q_{fg} = f^{-2} Q_g$ and $B_{fg} = f^{-1} B_g$, \Cref{eq:intro-J-flow-dim4} reduces to
\begin{equation*}
	f'(t) = -\frac{Q_0}{2} f(t)^{-1}, \qquad f(0) = 1,
\end{equation*}
with solution $f(t) = \sqrt{1-Q_0 t}$. 

Therefore, if $g_0$ is Ricci-flat, then $Q_0 = 0$ and $g_0$ is $J$-flat, hence a static solution. If $R_0 \neq 0$, then $Q_0 > 0$ and the solution shrinks homothetically, becoming extinct at time $T_{\max} = Q_0^{-1} = 24 R_0^{-2}$. In particular, on the round sphere $(\mathbb S^4, g_{\mathbb S^4})$ with sectional curvature $1$ one has $Q_0 = 6$, and the flow is a pure conformal shrinking with extinction time $\frac{1}{6}$. Notably, hyperbolic manifolds also shrink, because $Q_0$ depends on $R_0^2$.

\subsection{\texorpdfstring{$J$-flow on $S^2 \times \Sigma_\gamma$}{J-flow on S² × Σγ}}\label{subsec:sphere-hyperbolic}

We now turn to products of surfaces. The following curvature formulas will be used in all three product examples below.

\begin{lemma}\label[lemma]{lem:product-curvature}
	Let $(M_1^2, g_1)$ and $(M_2^2, g_2)$ be surfaces of constant Gauss curvature $\kappa_1$ and $\kappa_2$ respectively, and let $g = g_1 \oplus g_2$ on $M = M_1 \times M_2$. Then
	\begin{equation*}
		Q_g = -\frac{1}{3} \left(\kappa_1^2 - 4 \kappa_1 \kappa_2 + \kappa_2^2 \right),
	\end{equation*}
	the Bach tensor is diagonal with
	\begin{equation*}
		B_g = \left( \frac{\kappa_1^2 - \kappa_2^2}{6} \right)  g_1 \oplus \left( -\frac{\kappa_1^2 - \kappa_2^2}{6} \right) g_2,
	\end{equation*}
	and consequently
	\begin{equation}\label{eq:product-J}
		J_g = \left( -\frac{\kappa_1(\kappa_1 - 2\kappa_2)}{6} \right) g_1  \oplus \left( -\frac{\kappa_2(\kappa_2 - 2 \kappa_1)}{6} \right) g_2 .
	\end{equation}
\end{lemma}

\begin{proof}
	Fix a local orthonormal frame $\{e_1, e_2, e_3, e_4\}$ with $e_1, e_2$ tangent to $M_1$ and $e_3, e_4$ tangent to $M_2$. Up to the symmetries of $\Rm$, the only non-vanishing curvature components are $R_{1221} = \kappa_1$ and $R_{3443} = \kappa_2$, whence
	\begin{equation*}
		\Ric = \diag(\kappa_1, \kappa_1, \kappa_2, \kappa_2), \qquad R = 2(\kappa_1 + \kappa_2), \qquad | \Ric |^2 = 2 \kappa_1^2 + 2 \kappa_2^2 .
	\end{equation*}
	Since $R$ is constant, formula \Cref{eq:Q-dim4-intro} gives
	\begin{equation*}
		Q_g = -\frac{1}{2} \left(2 \kappa_1^2 + 2 \kappa_2^2 \right) + \frac{1}{6} \cdot 4( \kappa_1 + \kappa_2)^2 = -\frac{1}{3} \left(\kappa_1^2 - 4 \kappa_1 \kappa_2 + \kappa_2^2 \right).
	\end{equation*}
	Next we compute the Weyl tensor from the decomposition in \Cref{eq:Rm-decomp}. The traceless Ricci tensor is 
	\begin{equation*}
		E = \diag(\frac{\kappa_1 - \kappa_2}{2}, \frac{\kappa_1 - \kappa_2}{2}, - \frac{\kappa_1 - \kappa_2}{2}, -\frac{\kappa_1 - \kappa_2}{2}).
	\end{equation*}
	A direct evaluation of the Kulkarni--Nomizu products in the frame above yields
	\begin{equation}\label{eq:product-Weyl}
		W_{1221} = W_{3443} = \frac{\kappa_1 + \kappa_2}{3}, \qquad W_{1331} = W_{1441} = W_{2332} = W_{2442} = -\frac{\kappa_1 + \kappa_2}{6},
	\end{equation}
	with all components not obtained from these by the symmetries of $W$ equal to zero. Because $g$ is a product of constant-curvature metrics it is locally symmetric, so $\nabla W = 0$ and the classical formula \Cref{eq:Bach-classical} reduces to the zeroth-order term
	\begin{equation*}
		B_{ij} = \frac{1}{2} \R^{kl} W_{iklj} .
	\end{equation*}
	Using \Cref{eq:product-Weyl},
	\begin{equation*}
		B_{11} = \frac{1}{2} \left( \kappa_1 W_{1221} + \kappa_2 W_{1331} + \kappa_2 W_{1441} \right) = \frac{1}{2}( \kappa_1 + \kappa_2) \left(\frac{\kappa_1}{3} - \frac{\kappa_2}{3} \right) = \frac{\kappa_1^2 - \kappa_2^2}{6},
	\end{equation*}
	and the analogous computation gives $B_{33} = -B_{11}$. As consistency checks, $B$ vanishes when $\kappa_1 = \kappa_2$ (the Einstein case) and when $\kappa_1 = -\kappa_2$, where \Cref{eq:product-Weyl} shows $W \equiv 0$, so the metric is locally conformally flat and scalar-flat. Finally, \Cref{eq:J-dim4} gives, on the first factor,
	\begin{equation*}
		J_{11} = \frac{Q_g}{4} - \frac{B_{11}}{2} = -\frac{\kappa_1^2 - 4\kappa_1 \kappa_2 + \kappa_2^2}{12} - \frac{\kappa_1^2 - \kappa_2^2}{12} = -\frac{\kappa_1( \kappa_1 - 2 \kappa_2)}{6},
	\end{equation*}
	and similarly on the second factor, proving \Cref{eq:product-J}. 
\end{proof}

Let $\Sigma_\gamma$ be a closed oriented surface of genus $\gamma \ge 2$ carrying a hyperbolic metric $g_{\mathrm{hyp}}$ of curvature $-1$, and let $g_{S^2}$ denote the round metric of curvature $+1$ on $S^2$. On $M = S^2 \times \Sigma_\gamma$, consider the family
\begin{equation*}
	g_{A,B} := A g_{S^2} \oplus B g_{\mathrm{hyp}},
\end{equation*}
where $A(t), B(t) > 0$, with initial values $A_0 = A(0)$ and $B_0 = B(0)$. The factors have Gauss curvatures $\kappa_1 = \frac{1}{A}$ and $\kappa_2 = - \frac{1}{B}$.

The Euler characteristic and signature are
\begin{equation*}
	\chi(M) = \chi(S^2) \chi(\Sigma_\gamma) = 2( 2 - 2 \gamma) < 0, \qquad \tau(M) = 0,
\end{equation*}
which violate the Hitchin--Thorpe necessary condition $\chi \ge \frac{3}{2} |\tau|$ for Einstein four-manifolds~\cite{Hitchin1974,Thorpe1969}. Thus $M$ admits no Einstein metric.

Under the Ricci flow, the family $g_{A, B}$ evolves by $A' = -2$ and $B' = 2$, so the spherical factor collapses at time $\frac{1}{2}A_0$. Since $\frac{A}{B} \to 0$, no common homothetic normalization converges to a nondegenerate metric on the same compact product. This provides the comparison with the $J$-flow within this family.

We now examine the $J$-flow. Substituting the product metric into \Cref{eq:product-J} shows that \Cref{eq:intro-J-flow} preserves this family and reduces to the system
\begin{equation}\label{eq:product-ODE}
	A' = \frac{\kappa_1( \kappa_1 - 2 \kappa_2)}{3} A = \frac{B + 2A}{3AB}, \qquad
	B' = \frac{\kappa_2( \kappa_2 - 2 \kappa_1)}{3} B = \frac{A + 2B}{3AB} .
\end{equation}
Both right-hand sides are positive. Since $A \ge A_0$ and $B \ge B_0$, they also satisfy $A' \le (3A_0)^{-1} + 2 (3B_0)^{-1}$ and $B' \le (3B_0)^{-1} + 2 (3A_0)^{-1}$. Thus neither factor reaches zero or infinity in finite time, and the ODE solution exists for all time. The behavior of the shape is governed by the ratio $\rho := \frac{A}{B}$, with initial value $\rho_0 := \rho(0) = \frac{A_0}{B_0}$. It satisfies
\begin{equation}\label{eq:rho-ODE}
	\rho' = \frac{A'B - AB'}{B^2} = \frac{B^2 - A^2}{3A B^3} = \frac{1 - \rho^2}{3 \rho B^2} .
\end{equation}
Thus $\rho$ increases when $\rho < 1$ and decreases when $\rho > 1$, and stays in $[\min \{\rho_0, 1\}, \max\{\rho_0, 1\}]$. From \Cref{eq:product-ODE},
\begin{equation*}
	(A^2)' = \frac{2}{3}(1 + 2\rho), \qquad (B^2)' = \frac{2}{3} \left(1 + \frac{2}{\rho} \right).
\end{equation*}
In particular, $B^2$ is bounded above and below by positive affine functions of $t$ since $\rho$ is bounded above and below, and we have
\begin{equation*}
	\sigma(t) :=  \int_0^t \frac{1}{3 B(s)^2} ds \to \infty
\end{equation*}
as $t \to \infty$. Then $d \rho / d \sigma = \rho^{-1} - \rho$. Since $\sigma \to \infty$ and $\rho$ stays in a positive compact interval, its monotone limit must be $1$. Consequently $(A^2)' \to 2$ and $(B^2)' \to 2$, hence
\begin{equation*}
	A(t)^2 = 2 t \left(1 + o(1) \right), \qquad B(t)^2 = 2 t \left(1 + o(1) \right), \qquad \frac{A(t)}{B(t)} \longrightarrow 1 ,
\end{equation*}
and the rescaled metrics converge:
\begin{equation*}
	\frac{1}{\sqrt{2t}} g(t) \longrightarrow \tilde g := g_{S^2} \oplus g_{\mathrm{hyp}} \qquad \text{as } t \to \infty .
\end{equation*}
Since the factor metrics are fixed, convergence of the two coefficients implies convergence in every $C^k$-norm with respect to any fixed smooth background metric.

By \Cref{lem:product-curvature}, the limiting metric $\tilde g$ is scalar-flat and locally conformally flat, hence Bach-flat, and has constant $Q$-curvature $Q_{\tilde g} = -2$. Therefore
\begin{equation*}
	J_{\tilde g} = -\frac{1}{2} \tilde g .
\end{equation*}
Thus $\tilde g$ is $J$-Einstein. For initial values $A_0 = B_0 = a > 0$, the diagonal solution of \Cref{eq:product-ODE} is the homothetically expanding solution
\begin{equation*}
	g(t)=\sqrt{a^2 + 2 t} \tilde g.
\end{equation*}

It is also useful to describe the volume-normalized $J$-flow, with time parameter $\theta$:
\begin{equation}\label{eq:normalized-J-flow}
	\frac{\partial}{\partial\theta} g = -2 J_g + \frac{\bar Q_g}{2} g, \qquad \bar Q_g := \frac{1}{\Vol_g(M)} \int_M Q_g dv_g.
\end{equation}
Since $\tr_g J_g = Q_g$ in dimension four, this equation preserves total volume. For the product metrics, it therefore preserves $A B = A_0 B_0 =: P$. The common conformal term cancels from the ratio equation, giving
\begin{equation}\label{eq:normalized-product-ratio}
	\frac{d \rho}{d \theta} = \frac{1 - \rho^2}{3 P}.
\end{equation}
Setting $z = (\rho - 1)/(\rho + 1)$ gives $dz / d\theta = -2 z / (3 P)$, and hence
\begin{equation}\label{eq:normalized-product-ratio-solution}
	\frac{\rho(\theta) - 1}{\rho(\theta) + 1} = \frac{\rho_0 - 1}{\rho_0 + 1} e^{- \frac{2\theta}{3P}}.
\end{equation}
The right-hand side has absolute value less than one for all $\theta\ge0$. Thus $\rho$ stays positive and bounded, exists for all normalized time, and converges exponentially to $1$. Since $A = \sqrt{P \rho}$ and $B = \sqrt{P / \rho}$, the normalized metrics converge exponentially in every $C^k$-norm to the stationary metric
\begin{equation*}
	\sqrt{P} (g_{S^2} \oplus g_{\mathrm{hyp}}).
\end{equation*}
This limit equals $\tilde g$ when $P = 1$.

\subsection{\texorpdfstring{$J$-flow on $S^2 \times S^2$}{J-flow on S² × S²}}\label{subsec:sphere-sphere}

On $S^2 \times S^2$, consider $g_{A,B} = A g_{S^2} \oplus B g_{S^2}$, with factors of Gauss curvature $\kappa_1 = A^{-1}$ and $\kappa_2 = B^{-1}$. By \Cref{lem:product-curvature}, the unnormalized $J$-flow equations are
\begin{equation*}
	A' = \frac{B-2A}{3AB}, \qquad B' = \frac{A-2B}{3AB}.
\end{equation*}
The ratio $\rho := A/B$ still satisfies \Cref{eq:rho-ODE}. Under the volume-normalized flow, $AB = P = A_0B_0$ is preserved, and $\rho$ satisfies \Cref{eq:normalized-product-ratio}. Its explicit solution \Cref{eq:normalized-product-ratio-solution} shows that the normalized flow is immortal and converges exponentially in every $C^k$-norm to the symmetric Einstein metric $\sqrt P(g_{S^2} \oplus g_{S^2})$. In particular, this metric is asymptotically stable within the product family of fixed volume. For comparison, the volume-normalized Ricci flow preserves $P$ but satisfies $d\rho/d\theta = 2(\rho-1)/B$, so its symmetric product is unstable within the same family.

The corresponding unnormalized $J$-flow becomes extinct in finite time while its ratio tends to $1$. To see this without assuming the ratio limit, write $g(t(\theta)) = f(\theta) \bar g(\theta)$, where $\bar g$ is the normalized solution, and set
\begin{equation*}
	\frac{dt}{d\theta} = f^2, \qquad \frac{df}{d\theta} = -\frac{1}{2} \bar Q(\theta) f, \qquad f(0)=1.
\end{equation*}
The normalized limit has $\bar Q \to 2/(3P)>0$. Thus $f$ decays exponentially for large $\theta$, $t(\theta)$ approaches a finite extinction time, and $A/B \to 1$ there.

\subsection{\texorpdfstring{$J$-flow on $S^2 \times \mathbb T^2$}{J-flow on S² × T²}}\label{subsec:sphere-torus}

The following product solution shows that decreasing curvature energy need not yield convergence of the unnormalized flow.

Let $g_{S^2}$ have Gauss curvature one and area $4 \pi$, and let $h$ be a flat metric on $\mathbb T^2$, whose volume with respect to $h$ is denoted by $\Vol(\mathbb T^2)$. For $A_0 > 0$ set
\begin{equation}\label{eq:sphere-torus-flow}
	g(t) = A(t) g_{S^2} \oplus h, \qquad A(t) = \sqrt{A_0^2 + \frac{2}{3} t}.
\end{equation}
This is an immortal $J$-flow with positive initial Yamabe constant, decreasing total curvature energy, and $\mathcal J(g_0)<0$. Although both $\|Q_{g(t)}\|_2$ and $\|J_{g(t)}\|_2$ tend to zero, neither is integrable in time, and the volume diverges.

Indeed, substituting $\kappa_1=A^{-1}$ and $\kappa_2=0$ into \Cref{eq:product-J} gives $J=-(6A^2)^{-1}(A g_{S^2})\oplus0$. Thus $A'=1/(3A)$ and $\partial_t h=0$, proving \Cref{eq:sphere-torus-flow}. Direct calculation gives
\begin{equation}\label{eq:sphere-torus-scalar}
	R= \frac{2}{A}, \qquad Q = - \frac{1}{3 A^2}, \qquad \Vol_g(M) = 4 \pi A \Vol(\mathbb T^2),
\end{equation}
\begin{equation}\label{eq:sphere-torus-energy}
	\int_M |\Rm|^2 dv_g = 16 \pi \frac{\Vol(\mathbb T^2)}{A}, \qquad \mathcal J(g) = -4 \pi \frac{\Vol(\mathbb T^2)}{3A}, \qquad \|Q\|_2^2 = \|B\|_2^2 = 4\pi \frac{\Vol(\mathbb T^2)}{9A^3}.
\end{equation}
Here $B=(6A^2)^{-1}(A g_{S^2}\oplus(-h))$, which verifies the last norm identity and $\mathcal J'=\|B\|_2^2$. The formula for $J$ also gives
\begin{equation}\label{eq:sphere-torus-J-norm}
	\|J\|_2^2=\frac{1}{2}\|Q\|_2^2
	=\frac{2\pi\Vol(\mathbb T^2)}{9A^3}.
\end{equation}

Since $R > 0$, the conformal Laplacian is positive, and the initial Yamabe constant is positive. More explicitly, a Sobolev inequality for each fixed compact metric gives positivity of its Yamabe quotient. Testing this quotient with $u = 1$ yields
\begin{equation*}
	0 < Y(M,[g(t)]) \le \frac{2}{A} \sqrt{4 \pi A\Vol(\mathbb T^2)} \longrightarrow 0.
\end{equation*}
Thus positivity at each finite time does not give a uniform positive Yamabe lower bound. The hypothesis $\mathcal J(g_0)>0$ in \Cref{prop:Yamabe-propagation} fails in this example.

The volume and the spherical diameter diverge. Hence no sequence $g(t_i)$ with $t_i\to\infty$ can converge smoothly to a positive definite metric on the compact manifold, even after pullback by diffeomorphisms, which preserve total volume. Moreover, \Cref{eq:sphere-torus-energy,eq:sphere-torus-J-norm} show that both $\|Q\|_2$ and $\|J\|_2$ decay at rate $t^{-3/4}$, so
\begin{equation*}
	\int_0^\infty\|Q_{g(t)}\|_2\,dt=\infty,
	\qquad \int_0^\infty\|J_{g(t)}\|_2\,dt=\infty.
\end{equation*}
In particular, the time-integrability hypotheses in \Cref{theorem:E,theorem:F} do not follow from decay of these norms alone.

At fixed $A_0$, choosing $\Vol(\mathbb T^2)$ small makes the initial curvature energy arbitrarily small across this family, and the energy subsequently decreases. Thus initial Yamabe positivity and a universal absolute small-energy bound do not by themselves imply a uniform positive Yamabe lower bound, smooth subsequential convergence, or time integrability of $\|Q\|_2$ and $\|J\|_2$. The example clarifies the role of the additional hypotheses in \Cref{theorem:D,theorem:E,theorem:F}.

\section{Short-time existence}\label{sec:short-time}

In this section we establish the short-time existence and uniqueness of solutions to the $J$-flow on closed manifolds of dimension $n \ge 4$. Because the right-hand side is of fourth order in the metric, the classical first-order DeTurck gauge used for the Ricci flow is no longer sufficient. We first explain the obstruction and then introduce a biharmonic DeTurck gauge that recover strict parabolicity. 

\subsection{Principal symbol of the ungauged flow}

Using the explicit formula \Cref{eq:J-full}, the highest-order terms in the right-hand side of \Cref{eq:intro-J-flow} are
\begin{equation}\label{eq:J-principal}
	\frac{2}{(n-2)^2} (\Delta \Ric)_{ij} + \frac{n^2 - 10n + 12}{4 (n-1)^2 (n-2)} \nabla_i \nabla_j R + \frac{(3n-4) (n-4)}{4 (n-1)^2 (n-2)^2} g_{ij} \Delta R.
\end{equation}
To analyze ellipticity, we linearize the ungauged operator at a fixed metric $g$. Let $h \in S_2(M)$ be a variation. For brevity set
\begin{equation}\label{eq:alpha-beta-gamma}
	\alpha_n  := \frac{n^3 - 8n^2 + 24n - 20}{4 (n-1)^2 (n-2)^2}, \qquad
	\beta_n := \frac{n^2 - 10n + 12}{4 (n-1)^2 (n-2)}, \qquad
	\gamma_n  := \frac{(3n-4) (n-4)}{4(n-1)^2 (n-2)^2}.
\end{equation}
A direct computation using \Cref{eq:Ric-linearization,eq:R-linearization,eq:J-full} gives the following expression for the principal part of the linearized ungauged operator:

\begin{lemma}\label[lemma]{lem:ungauged-linearization}
	The linearization of the ungauged $J$-flow operator $-2 J_g$ at $g$ has principal part
	\begin{equation}\label{eq:ungauged-linearization}
		\begin{split}
			\left(D( -2J_g) \cdot h \right)_{ij}
			=& -\frac{1}{(n-2)^2} \left( \Delta^2 h_{ij} + \nabla_i \Delta(\delta h)_j + \nabla_j \Delta(\delta h)_i \right) - \alpha_n \nabla_i \nabla_j \Delta \tr_gh + \beta_n \nabla_i \nabla_j \delta^2h \\
			&- \gamma_n g_{ij} \left( \Delta^2 \tr_g h - \Delta \delta^2 h \right) + \text{l.o.t.}
		\end{split}
	\end{equation}
\end{lemma}

\begin{proof}
	The highest-order part of $-2 J_g$ is given by \Cref{eq:J-principal}. We linearize each of the three fourth-order summands separately, retaining only terms of differential order four in $h$. Throughout, commutators of covariant derivatives with the linearization produce curvature contractions against at most three derivatives of $h$, and are therefore absorbed into the remainder.

	We begin with the Hessian of the scalar curvature. The linearization of scalar curvature is~\Cref{eq:R-linearization}. Discarding the zeroth-order term $\Ric\cdot h$, we have
	\begin{equation*}
		D R_g \cdot h = -\Delta \tr_g h + \delta^2 h + \text{l.o.t.},
	\end{equation*}
	which yields
	\begin{equation*}
		D(\nabla_i \nabla_j R) \cdot h = \nabla_i \nabla_j (DR_g \cdot h) = -\nabla_i \nabla_j \Delta \tr_gh + \nabla_i \nabla_j \delta^2 h + \text{l.o.t.}
	\end{equation*}
	Next we treat the rough Laplacian of the Ricci tensor. From~\Cref{eq:Ric-linearization},
	\begin{equation*}
		(D \Ric_g \cdot h)_{ij} = -\frac{1}{2} \left( \Delta h_{ij} + \nabla_i(\delta h)_j + \nabla_j(\delta h)_i + \nabla_i \nabla_j \tr_g h \right) + \text{l.o.t.}
	\end{equation*}
	Hence
	\begin{equation*}
		\left(D(\Delta \Ric) \cdot h \right)_{ij}
		= \Delta \left(D \Ric_g \cdot h \right)_{ij} + \text{l.o.t.} 
		= -\frac{1}{2} \left( \Delta^2h_{ij} + \nabla_i \Delta(\delta h)_j + \nabla_j \Delta(\delta h)_i + \nabla_i \nabla_j \Delta \tr_g h \right) + \text{l.o.t.},
	\end{equation*}
	where we have used $\Delta \nabla_i (\delta h)_j = \nabla_i \Delta (\delta h)_j + \text{l.o.t.}$. It remains to linearize the pure-trace term $g\Delta R$, the product rule gives
	\begin{equation*}
		D(g \Delta R) \cdot h = (D g \cdot h) \Delta R + g (D (\Delta R) \cdot h).
	\end{equation*}
	The first summand is zeroth order in the derivatives of $h$. For the second, the variation of the scalar Laplacian applied to the fixed function $R$ produces at most first-order terms in $h$, while
	\begin{equation*}
		D(\Delta R) \cdot h = \Delta (D R_g \cdot h) + \text{l.o.t.} = -\Delta^2 \tr_gh + \Delta \delta^2 h + \text{l.o.t.}
	\end{equation*}
	Consequently
	\begin{equation*}
		\left(D( g \Delta R ) \cdot h \right)_{ij} = g_{ij} \left( -\Delta^2 \tr_g h + \Delta \delta^2 h \right) + \text{l.o.t.}
	\end{equation*}
	Collecting these three contributions against the coefficients in~\Cref{eq:J-principal}, we obtain
	\begin{equation*}
		\begin{split}
			\left(D( -2J_g) \cdot h \right)_{ij}
			=& \frac{2}{(n-2)^2} \left(D (\Delta \Ric) \cdot h \right)_{ij} + \beta_n D(\nabla_i \nabla_j R) \cdot h + \gamma_n D(g_{ij} \Delta R) \cdot h + \text{l.o.t.} \\
			=& -\frac{1}{(n-2)^2} \left( \Delta^2 h_{ij} + \nabla_i \Delta(\delta h)_j + \nabla_j \Delta(\delta h)_i \right) - \left(\frac{1}{(n-2)^2} + \beta_n \right) \nabla_i \nabla_j \Delta \tr_gh \\
			&+ \beta_n \nabla_i \nabla_j \delta^2h - \gamma_n g_{ij} \left( \Delta^2 \tr_g h - \Delta \delta^2 h \right) + \text{l.o.t.}
		\end{split}
	\end{equation*}
	The algebraic identity $\frac{1}{(n-2)^2} + \beta_n = \alpha_n$ yields~\Cref{eq:ungauged-linearization}.
\end{proof}

\subsection{The biharmonic DeTurck gauge}

To overcome the mixed-term obstruction, fix a smooth background metric $\tilde g$ on $M$. The DeTurck vector field of $g$ relative to $\tilde g$ is
\begin{equation*}
	V^k := g^{pq} \left( \tilde \Gamma^k_{pq} - \Gamma^k_{pq}(g) \right),
\end{equation*}
where $\Gamma^k_{pq} (g)$ denotes the Christoffel symbols of $g$, and we write $V_j := g_{jk} V^k$ for the associated covector. Its linearization is
\begin{equation*}
	(D V \cdot h)_j = (\delta h)_j + \frac{1}{2} \nabla_j \tr_g h + \text{l.o.t.}
\end{equation*}

\begin{definition}[Biharmonic DeTurck gauge]\label{def:biharmonic-gauge}
	The biharmonic DeTurck gauge vector field $X$ is defined by the covector
	\begin{equation}\label{eq:biharmonic-X}
		X_j := \frac{1}{(n-2)^2} \Delta_g V_j - \frac{\beta_n}{2} \nabla_j R_g .
	\end{equation}
	The biharmonic-gauged $J$-flow is
	\begin{equation}\label{eq:J-DeTurck}
		\frac{\partial}{\partial t} g = -2 J_g + \mathcal{L}_X g .
	\end{equation}
\end{definition}

The coefficients in \Cref{eq:biharmonic-X} are chosen precisely so that the following cancellation occurs:

\begin{proposition}\label{prop:gauge-linearization}
	The linearization of the gauge term $\mathcal{L}_X g$ at $g$ satisfies
	\begin{equation}\label{eq:gauge-linearization}
		\left(D (\mathcal{L}_X g) \cdot h \right)_{ij} = \frac{1}{(n-2)^2} \left( \nabla_i \Delta (\delta h)_j + \nabla_j \Delta (\delta h)_i \right) + \alpha_n \nabla_i \nabla_j \Delta \tr_g h - \beta_n \nabla_i \nabla_j \delta^2 h + \text{l.o.t.}
	\end{equation}
\end{proposition}

\begin{proof}
	We first compute the linearization of the gauge vector field $X$. Applying the rough Laplacian to the linearization of $V$ and commuting derivatives gives
	\begin{equation*}
		\Delta((DV \cdot h)_j) = \Delta(\delta h)_j + \frac{1}{2} \nabla_j \Delta \tr_gh + \text{l.o.t.}
	\end{equation*}
	The scalar-curvature correction contributes
	\begin{equation*}
		-\frac{\beta_n}{2} \nabla_j(D R_g \cdot h) = \frac{\beta_n}{2} \nabla_j \Delta \tr_g h - \frac{\beta_n}{2} \nabla_j \delta^2 h + \text{l.o.t.}
	\end{equation*}
	Adding these and using $\alpha_n = \frac{1}{(n-2)^2} + \beta_n$, we obtain
	\begin{equation*}
		\begin{split}
			(DX \cdot h)_j
			 & = \frac{1}{(n-2)^2} \left( \Delta(\delta h)_j + \frac{1}{2} \nabla_j \Delta \tr_g h \right) + \frac{\beta_n}{2} \nabla_j \Delta \tr_g h - \frac{\beta_n}{2} \nabla_j \delta^2 h + \text{l.o.t.} \\
			 & = \frac{1}{(n-2)^2} \Delta(\delta h)_j + \frac{\alpha_n}{2} \nabla_j \Delta \tr_g h - \frac{\beta_n}{2} \nabla_j \delta^2 h + \text{l.o.t.}
		\end{split}
	\end{equation*}
	Passing to the Lie derivative, recall that $(\mathcal{L}_X g)_{ij} = \nabla_i X_j + \nabla_j X_i$. Since the variation of the Levi-Civita connection contributes only first-order terms in $h$, one has
	\begin{equation*}
		\left( D(\mathcal{L}_X g) \cdot h \right)_{ij} = \nabla_i (D X \cdot h)_j + \nabla_j(D X \cdot h)_i + \text{l.o.t.}
	\end{equation*}
	Substituting the expression for $D X \cdot h$ obtained above and the result follows.
\end{proof}

\begin{proposition}\label{thm:linearized-symbol}
	The linearization of the biharmonic-gauged $J$-flow operator at $g$ has principal part
	\begin{equation}\label{eq:gauged-principal}
		\left(D( -2J_g + \mathcal{L}_X g) \cdot h \right)_{ij} = -\frac{1}{(n-2)^2} \Delta^2 h_{ij} - \gamma_n g_{ij} \Delta^2 \tr_g h + \gamma_n g_{ij} \Delta \delta^2 h + \text{l.o.t.}
	\end{equation}
	In particular, the system \Cref{eq:J-DeTurck} is strictly parabolic on the bundle of symmetric $2$-tensors.
\end{proposition}

\begin{proof}
	Adding~\Cref{eq:ungauged-linearization,eq:gauge-linearization}, the mixed fourth-order terms cancel identically. What remains is precisely~\Cref{eq:gauged-principal}. It remains to check that the principal symbol $\sigma_{\xi}$ is negative definite. 
	
	Let $\xi \in T_p^*M \setminus \{0\}$ and let $H \in \mathrm{Sym}^2 (T_p^*M)$. With the convention $\Delta = g^{ij} \nabla_i \nabla_j$, one has
	\begin{equation*}
		\sigma_{\xi} (\Delta) = -|\xi|^2, \qquad \sigma_{\xi} (\Delta^2) = |\xi|^4.
	\end{equation*}
	The divergence satisfies
	\begin{equation*}
		\left( \sigma_{\xi} (\delta) H \right)_j = -i \xi^p H_{pj}, \quad \sigma_{\xi}(\delta^2) H = -\xi^p \xi^q H_{pq}, \quad \sigma_{\xi}(\Delta \delta^2) H = | \xi|^2 \xi^p \xi^q H_{pq}.
	\end{equation*}
	The principal symbol of the right-hand side of~\Cref{eq:gauged-principal} is consequently the endomorphism
	\begin{equation*}
		\left(\sigma_{\xi}(H) \right)_{ij} = -\frac{1}{(n-2)^2} | \xi|^4 H_{ij} - \gamma_n| \xi|^4 (\tr_g H) g_{ij} + \gamma_n| \xi|^2 \xi^p \xi^q H_{pq} g_{ij}.
	\end{equation*}
	We claim that $\langle \sigma_{\xi} (H), H \rangle_g < 0$ whenever $H \neq 0$. Expanding,
	\begin{equation*}
		\langle \sigma_{\xi}(H),H \rangle_g = -\frac{1}{(n-2)^2} | \xi|^4 |H|_g^2 - \gamma_n | \xi|^4 (\tr_g H)^2 + \gamma_n | \xi|^2 \xi^p \xi^q H_{pq} (\tr_g H).
	\end{equation*}
	If $\tr_g H = 0$, the last two terms vanish and the first is strictly negative. If $H = \mu g$ is pure trace, then $\tr_g H = n \mu$, $|H|_g^2 = n \mu^2$, and $\xi^p \xi^q H_{pq} = \mu |\xi|^2$, so $\sigma_{\xi} (\mu g) = \lambda_n (\xi) (\mu g)$, where
	\begin{equation}\label{eq:pure-trace-eigenvalue}
		\lambda_n(\xi) = -\left(\frac{1}{(n-2)^2} + (n - 1) \gamma_n \right) | \xi|^4 = -\frac{4(n-1) + (3n-4)(n-4)}{4(n-1)(n-2)^2} | \xi|^4 = -\frac{3}{4(n-1)} | \xi|^4,
	\end{equation}
	which is strictly negative for every $n\ge4$ and $\mu\neq0$. In particular $\gamma_4=0$ yields $\lambda_4 = -\frac{|\xi|^4}{4}$. For a general $H$, the cross term satisfies
	\begin{equation*}
		\left| \gamma_n| \xi|^2 \xi^p \xi^q H_{pq} (\tr_g H) \right| \le \gamma_n| \xi|^4 |H|_g | \tr_g H|,
	\end{equation*}
	and Young's inequality yields
	\begin{equation*}
		-\gamma_n | \xi|^4 (\tr_g H)^2 + \gamma_n | \xi|^4 |H|_g  |\tr_g H| \le \frac{1}{2} \gamma_n| \xi|^4 |H|_g^2.
	\end{equation*}
	Hence
	\begin{equation*}
		\langle \sigma_{\xi}(H), H \rangle_g \le - \left( \frac{1}{(n-2)^2} - \frac{1}{2} \gamma_n \right) | \xi|^4 |H|_g^2.
	\end{equation*}
	For $n = 4$ one has $\gamma_4 = 0$. For $n \ge 5$,
	\begin{equation*}
		\frac{1}{(n-2)^2} - \frac{1}{2} \gamma_n = \frac{8(n-1)^2-(3n-4)(n-4)}{8(n-1)^2(n-2)^2} = \frac{5n^2-8}{8(n-1)^2(n-2)^2} > 0.
	\end{equation*}
	Thus $\langle \sigma_{\xi}(H),H \rangle_g \le - c_n |\xi|^4 |H|_g^2$ with $c_n > 0$, so $\sigma_{\xi}$ is negative definite. Therefore~\Cref{eq:J-DeTurck} is a strictly parabolic fourth-order system.
\end{proof}

\subsection{Short-time existence and uniqueness}

With strict parabolicity of the gauged equation, existence and uniqueness of a short-time solution follow from the Schauder theory for quasilinear parabolic systems of order four~\cite{LadyzhenskayaSolonnikovUralceva1968}. The same short-time existence may equivalently be obtained in the smooth category from Hamilton's Nash--Moser inverse function theorem~\cite{Hamilton1981}.

\begin{proof}[Proof of \Cref{theorem:A}]
	Fix a smooth background metric $\tilde g = g_0$. The operator $-2 J_g + \mathcal L_{X(g, \tilde g)}g$ is a smooth quasilinear fourth-order operator with strongly parabolic linearization by \Cref{thm:linearized-symbol}. The parabolic construction in the proof of Bahuaud--Helliwell~\cite[Theorem~A]{BahuaudHelliwell2011} therefore provides a solution $\bar g(t)$ of \Cref{eq:J-DeTurck}. Equivalently, one may use fourth-order parabolic Schauder theory in $C^{4 + \alpha, 1 + \alpha/4}(M \times [0,T])$, with the initial trace prescribed. Positivity persists for short time and parabolic regularity gives smoothness for smooth initial data.

	Let $\varphi_t$ solve
	\begin{equation*}
		\partial_t \varphi_t = -X(\bar g(t), \tilde g) \circ \varphi_t, \qquad \varphi_0 = \mathrm{id}.
	\end{equation*}
	For a time-dependent tensor $h_t$ and a diffeomorphism satisfying $\partial_t \psi_t = Y_t \circ \psi_t$, we have
	\begin{equation}\label{eq:pullback-pushforward}
		\begin{split}
		\frac d{dt}( \psi_t^* h_t) &= \psi_t^* (\partial_t h_t + \mathcal L_{Y_t}h_t),\\
		\frac d{dt}((\psi_t)_* h_t) &= (\psi_t)_* (\partial_t h_t) - \mathcal L_{Y_t} ( (\psi_t)_* h_t).
		\end{split}
	\end{equation}
	Naturality of $J$ and the first identity imply that $g(t) = \varphi_t^* \bar g(t)$ solves $\partial_t g = -2 J_g$ with $g(0) = g_0$.

	For uniqueness we apply the intrinsic uniqueness theorem~\cite[Theorem~A]{BahuaudHelliwell2015}, verifying its coefficients. Its fourth-order class has the form
	\begin{equation*}
		T_g = c \left(\Delta \Ric + a( \Delta R) g - b \nabla^2 R \right) + T_{\le 3}, \qquad c>0,\quad a > -\frac{1}{2(n-1)}.
	\end{equation*}
	Here $T_{\le 3}$ is a polynomial natural tensor of lower differential order. By \Cref{eq:J-principal}, our choices are
	\begin{equation*}
		c = \frac2{(n-2)^2}, \qquad a = \frac{\gamma_n}{c} \ge 0, \qquad b = -\frac{\beta_n}{c},
	\end{equation*}
	and the remaining terms are quadratic curvature contractions. All hypotheses are satisfied for $n \ge 4 $. The theorem therefore gives uniqueness of the ungauged solution. 				
\end{proof}

\subsection{Comparison with the Bach flow}\label{subsec:Bach-comparison}

It is illuminating to contrast the preceding construction with the standard Bach flow. In any dimension $n \ge 4$ the Bach tensor has the schematic expansion~\cite{ChenLuQing2023}
\begin{equation}\label{eq:Bach-schematic}
	B_{ij} = \frac{1}{n-2} \Delta \R_{ij} - \frac{1}{2(n-1)(n-2)} (\Delta R) g_{ij} - \frac{1}{2(n-1)} \nabla_i \nabla_j R + \text{l.o.t.}
\end{equation}
With the present sign convention, the equation $\partial_tg=B_g$ has the forward sign in its nonconformal principal part, but remains degenerate in the conformal direction. After introducing the vector field
\begin{equation*}
	X^k = \frac{1}{2(n-2)} \Delta_g V^k + \frac{1}{4(n-1)}\nabla^kR,
\end{equation*}
the gauged operator has principal part
\begin{equation*}
	\left(D( B + \mathcal{L}_X g) \cdot h \right)_{ij} = -\frac{1}{2(n-2)} \Delta^2 h_{ij} + \frac{1}{2(n-1)(n-2)} g_{ij} \left[ \Delta^2(\tr_g h) - \Delta(\delta^2 h) \right].
\end{equation*}
In the pure-trace direction $h_{ij} = g_{ij}$ one has 
\begin{equation*}
	\sigma_{\xi} (\Delta^2 \tr_g h) = n |\xi|^4, \qquad\sigma_{\xi} (\Delta \delta^2 h) = |\xi|^4,
\end{equation*}
so the bracket contributes
\begin{equation*}
	\frac{n-1}{2(n-1)(n-2)} |\xi|^4 g = \frac{1}{2(n-2)} |\xi|^4 g,
\end{equation*}
which exactly cancels the rough biharmonic term, and $\sigma_{\xi}$ vanishes. Consequently the standard Bach flow is not strictly parabolic.

If we want to restore strict parabolicity, we need to add to the right-hand side the term
\begin{equation*}
	\frac{1}{2(n-1)(n-2)} (\Delta R) g_{ij}.
\end{equation*}
After the same DeTurck adjustment, the resulting modified equation has scalar biharmonic principal symbol and is strictly parabolic. 

For the $J$-flow the situation is different. The coefficient of $g \Delta R$ in the fully expanded $J$-tensor \Cref{eq:J-full} is $-\frac{\gamma_n}{2} (\Delta R) g_{ij}$, so that in $-2 J_g$ it becomes the positive constant $\gamma_n$ defined in \Cref{eq:alpha-beta-gamma}. This coefficient vanishes precisely at $n=4$, which explains why the four-dimensional $J$-flow has no residual pure-trace fourth-order term after gauging. For $n>4$ the coefficient is non-zero, as shown in~\Cref{eq:gauged-principal} and verified by the estimate on $\sigma_{\xi}$ in the proof of \Cref{thm:linearized-symbol}, the full symbol remains negative definite. In particular, the pure-trace eigenvalue~\Cref{eq:pure-trace-eigenvalue} remains strictly negative in every dimension $n \ge 4$. The mixed divergence/trace terms are removed by the biharmonic DeTurck gauge. 

Having established short-time existence, we turn in the next section to the differential inequalities that control the solution on its maximal existence interval. The principal analytic input is the evolution equation for the Riemann curvature tensor under~\Cref{eq:intro-J-flow}, from which we derive global and local $L^2$ energy inequalities and Bernstein--Bando--Shi smoothing estimates. The global smoothing estimates enter the continuation and convergence arguments in \Cref{sec:long-time,sec:J-estimates}, the local estimates provide a separate smoothing result. Local small-energy regularity and geometric compactness require additional arguments.

\section{Integral estimates and smoothing}\label{sec:integral-smoothing}
We begin with the standard variational formula for the Riemann curvature tensor. If $\frac{\partial}{\partial t} g = h$, then
\begin{equation}\label{eq:Rm-variation}
	\begin{split}
		\frac{\partial}{\partial t} R_{ijkl} = \frac{1}{2} \left(\nabla_i \nabla_k h_{jl} + \nabla_j \nabla_l h_{ik} - \nabla_i \nabla_l h_{jk} - \nabla_j \nabla_k h_{il} \right) + \frac{1}{2} \left( R_{ijkp} h_l{}^p - R_{ijlp} h_k{}^p \right) .
	\end{split}
\end{equation}
Throughout this section we write $P_{s,m}(\Rm)$ for a universal polynomial in the Riemann curvature tensor and its covariant derivatives, obtained by contracting $s$ curvature factors of total derivative order $m$. A generic term of $P_{s,m}(\Rm)$ thus has the form
\begin{equation*}
	\nabla^{i_1}\Rm * \cdots * \nabla^{i_s}\Rm, \qquad i_1 + \cdots + i_s = m,
\end{equation*}
where $*$ denotes a contraction with the metric.

\begin{lemma}\label[lemma]{lem:Rm-evolution}
	Let $g(t)$ be a solution of the $J$-flow on $M \times [0,T)$ with $n \ge 4$. Then the Riemann curvature tensor satisfies the schematic evolution equation
	\begin{equation}\label{eq:Rm-evolution}
		\frac{\partial}{\partial t} \Rm = -\frac{1}{(n-2)^2} \Delta^2 \Rm - \frac{\gamma_n}{2} g \odot \nabla^2 \Delta R + P_{2,2}(\Rm) + P_{3,0}(\Rm).
	\end{equation}
\end{lemma}

\begin{proof}
	Write $h = -2 J_g$. By \Cref{eq:J-principal} the fourth-order part of $h$ is
	\begin{equation*}
		h^{(4)}_{ij} = \frac{2}{(n-2)^2} (\Delta \Ric)_{ij} + \beta_n \nabla_i \nabla_j R + \gamma_n g_{ij} \Delta R .
	\end{equation*}
	We substitute this into the first group of terms in the variational formula \Cref{eq:Rm-variation}. A typical term is
	\begin{equation*}
		\frac{1}{2} \nabla_i \nabla_k h^{(4)}_{jl} = \frac{1}{(n-2)^2} \nabla_i \nabla_k (\Delta \Ric)_{jl} + \frac{\beta_n}{2} \nabla_i \nabla_k \nabla_j \nabla_l R + \frac{\gamma_n}{2} g_{jl}  \nabla_i \nabla_k \Delta R .
	\end{equation*}
	The first summand produces
	\begin{equation*}
		\frac{1}{(n-2)^2}\left(\nabla_i \nabla_k \Delta \R_{jl} + \nabla_j \nabla_l \Delta \R_{ik} - \nabla_i \nabla_l \Delta \R_{jk} - \nabla_j \nabla_k \Delta \R_{il} \right) + P_{2,2}(\Rm)  = - \frac{1}{(n-2)^2} \Delta^2R_{ijkl} + P_{2,2}(\Rm).
	\end{equation*}
	In the alternating combination appearing in \Cref{eq:Rm-variation}, the fully symmetric four-derivative parts cancel in pairs, and only commutator terms remain, in which curvature is contracted against at most two derivatives of $\Rm$. The same applies to the Hessian term $\nabla_i \nabla_j R$. By contrast, the pure-trace term $g_{jl} \Delta R$ carries only two derivatives of $R$, so its alternating combination does not vanish. With the Kulkarni--Nomizu convention~\Cref{eq:KN} one has
	\begin{equation*}
		g_{jl} \nabla_i \nabla_k \Delta R + g_{ik} \nabla_j \nabla_l \Delta R - g_{jk}\nabla_i \nabla_l \Delta R - g_{il} \nabla_j \nabla_k \Delta R = -(g \odot \nabla^2 \Delta R)_{ijkl}.
	\end{equation*}
	All remaining contributions are absorbed into $P_{2,2}(\Rm)$.

	The remaining terms in \Cref{eq:Rm-variation} are of two types. Those coming from $h^{(4)}$ but with fewer than four derivatives, and those coming from the lower-order part of $h$, are of the schematic types $P_{2,2}(\Rm)$ and $P_{3,0}(\Rm)$.
\end{proof}

\begin{corollary}\label{cor:Rm-derivative-evolution}
	For every integer $m \ge 0$ the $m$-th covariant derivative of the Riemann tensor satisfies
	\begin{equation}\label{eq:Rm-derivative-evolution}
		\frac{\partial}{\partial t} \nabla^m \Rm = -\frac{1}{(n-2)^2} \Delta^2 \nabla^m \Rm - \frac{\gamma_n}{2} \nabla^m (g \odot \nabla^2 \Delta R) + P_{2,m+2}(\Rm) + P_{3,m}(\Rm).
	\end{equation}
\end{corollary}

\begin{proof}
	Applying $\nabla^m$ to \Cref{eq:Rm-evolution} gives
	\begin{equation*}
		\nabla^m P_{2,2}(\Rm) = P_{2,m+2}(\Rm), \qquad \nabla^m P_{3,0}(\Rm) = P_{3,m}(\Rm).
	\end{equation*}
	We first commute the spatial derivatives. The Ricci commutation identity gives
	\begin{equation*}
		\nabla^m \Delta^2 \Rm = \Delta^2 \nabla^m \Rm + P_{2,m+2} (\Rm).
   	\end{equation*}
	It remains to commute the time derivative through $\nabla^m$. For a general metric variation $\partial_t g = h$, one has
	\begin{equation*}
		[\partial_t, \nabla^m] \Rm = \sum_{i=0}^{m-1} \nabla^i( \partial_t \Gamma) * \nabla^{m-1-i} \Rm,
	\end{equation*}
	where the sum is empty when $m = 0$. Here $\partial_t \Gamma$
	is a tensor satisfying
	\begin{equation*}
		(\partial_t \Gamma)^k_{ij} = \frac{1}{2} g^{k l} \left( \nabla_i h_{jl} + \nabla_j h_{il} - \nabla_l h_{ij} \right).
	\end{equation*}
	Since $h = -2 J_g = P_{1,2} (\Rm) + P_{2,0} (\Rm)$, the quadratic curvature terms in this commutator have total	derivative order $(i+3) + (m-1-i) = m+2$, whereas the cubic curvature terms have total derivative order $(i+1) + (m-1-i) = m$. Consequently,
	\begin{equation*}
		[\partial_t, \nabla^m] \Rm = P_{2,m+2}(\Rm) + P_{3,m}(\Rm).
	\end{equation*}
	Combining these identities proves \Cref{eq:Rm-derivative-evolution}.
\end{proof}

\subsection{Global energy inequalities}
Set 
\begin{equation*}
	\mathcal E_j = \int_M |\nabla^j \Rm|^2 dv_g.
\end{equation*}
We use the following covariant interpolation inequality of Hamilton~\cite{hamilton1982three}, its general-exponent formulation is also recorded in~\cite[Proposition~9.9]{Bour2010}.

\begin{lemma}\label[lemma]{lem:Hamilton-interpolation}
Let $T$ be a smooth tensor of fixed rank on a closed Riemannian $n$-manifold. If $k \ge 1$, $ 2 \le r < \infty$, and $0 \le i \le k$, then
\begin{equation}\label{eq:Hamilton-interpolation}
	\| \nabla^i T \|_{p_i} \le C \| T \|_r^{1 - i/k} \| \nabla^k T \|_2^{i/k}, \qquad \frac{1}{p_i} = \frac{1 - i/k}{r} + \frac{i}{2k}.
\end{equation}
The constant depends only on $n,k,r$ and the rank of $T$.
\end{lemma}

The constants in this inequality arise from covariant integration by parts and H\"older's inequality. 

\begin{lemma}\label[lemma]{lem:adjacent-interpolation}
	For $j \ge 0$ and any smooth tensor $T$ on a closed manifold,
	\begin{equation}\label{eq:adjacent-Young}
	\|\nabla^{j+1} T\|_2^2 \le C(n) \|\nabla^j T\|_2 \|\nabla^{j+2} T\|_2 \le \varepsilon \|\nabla^{j+2} T \|_2^2 + C(n, \varepsilon) \|\nabla^j T\|_2^2.
	\end{equation}
\end{lemma}
\begin{proof}
	Integrate by parts once, apply Cauchy--Schwarz and use $|\Delta \nabla^j T | \le \sqrt{n} |\nabla^2 \nabla^j T |$.
\end{proof}

\begin{lemma}\label[lemma]{lem:polynomial-interpolation}
	Let $m \ge 0$ be an integer and suppose $|T| \le K$. For any scalar contractions of the indicated types and every $\varepsilon > 0$,
	\begin{equation}\label{eq:polynomial-Young}
		\left| \int_M P_{3, 2m+2} (T) + P_{4, 2m} (T) dv_g \right| \le \varepsilon \| \nabla^{m+2} T \|_2^2 + C_\varepsilon K^{m+2} \| T \|_2^2.
	\end{equation}
	The constant depends only on $n, m, \varepsilon$, the tensor rank and the finitely many contraction coefficients.
\end{lemma}

\begin{proof}
	In each complete contraction, if one factor has more than $m+2$ covariant derivatives, integrate its outermost derivative by parts onto the other factors, and repeat until every factor has at most $m + 2$ derivatives. Since the total derivative order is at most $2m+2$, the other factors cannot exceed $m+2$ derivatives during this process. No derivatives are commuted, and $\nabla g = 0$, so no curvature commutator terms arise. As $M$ is closed, there are no boundary terms. The resulting integrals are finite sums of terms of the forms
	\begin{equation*}
		\int_M \nabla^{i_1} T \ast \nabla^{i_2} T \ast \nabla^{i_3} T dv_g, 
		\quad \text{and} \quad
		\int_M \nabla^{i_1} T \ast \nabla^{i_2} T \ast \nabla^{i_3} T \ast \nabla^{i_4} T dv_g.
	\end{equation*}
	Use \Cref{eq:Hamilton-interpolation} with $k = m + 2$ and $r = m + 4$.
	The exponents simplify to
	\begin{equation*}
		\frac{1}{p_i} = \frac{1 - i/(m+2)}{m+4} + \frac{i}{2(m+2)} = \frac{i+2}{2(m+4)}, \qquad 0 \le i \le m+2.
	\end{equation*}
	The first term has derivative orders $i_1 + i_2+ i_3 = 2m + 2$ with $i_\nu \le m+2$, so 
	\begin{equation*}
		\frac{1}{p_{i_1}} + \frac{1}{p_{i_2}} + \frac{1}{p_{i_3}} = \frac{2m+2+6}{2(m+4)} = 1
	\end{equation*}
	therefore,
	\begin{equation*}
		\begin{split}
			\left| \int_M \nabla^{i_1} T \ast \nabla^{i_2} T \ast \nabla^{i_3} T dv_g \right|
			&\le C \left\| \nabla^{i_1} T \right\|_{p_{i_1}} \left\| \nabla^{i_2} T \right\|_{p_{i_2}} \left\| \nabla^{i_3} T \right\|_{p_{i_3}}\\
			&\le C \left\| T \right\|_{m+4}^{3- \frac{i_1 + i_2 + i_3}{m+2}} \| \nabla^{m+2} T \|_2^{\frac{i_1 + i_2 + i_3}{m+2}}\\
			&= C \left\| T \right\|_{m+4}^{\frac{m+4}{m+2}} \| \nabla^{m+2} T \|_2^{\frac{2m+2}{m+2}}\\
			&\le C(n,m,\epsilon) \left\| T \right\|_{m+4}^{m+4}  + \frac{\epsilon}{2} \| \nabla^{m+2} T \|_2^{2} 
		\end{split}
	\end{equation*}
	For the second term, $i_1 + \cdots + i_4 = 2 m$ with $i_\nu \le m+2$, and the corresponding sum is
	\begin{equation*}
		\frac{1}{p_{i_1}} + \frac{1}{p_{i_2}} + \frac{1}{p_{i_3}} + \frac{1}{p_{i_4}} = \frac{2m+8}{2(m+4)} = 1,
	\end{equation*}
	it gives
	\begin{equation*}
		\begin{split}
			\left| \int_M \nabla^{i_1} T \ast \nabla^{i_2} T \ast \nabla^{i_3} T \ast \nabla^{i_4} T dv_g \right|
			&\le C \left\| \nabla^{i_1} T \right\|_{p_{i_1}} \left\| \nabla^{i_2} T \right\|_{p_{i_2}} \left\| \nabla^{i_3} T \right\|_{p_{i_3}} \left\| \nabla^{i_4} T \right\|_{p_{i_4}}\\
			&\le C \left\| T \right\|_{m+4}^{4- \frac{i_1 + i_2 + i_3 + i_4}{m+2}} \| \nabla^{m+2} T \|_2^{\frac{i_1 + i_2 + i_3 + i_4}{m+2}}\\
			&= C \left\| T \right\|_{m+4}^{\frac{2m+8}{m+2}} \| \nabla^{m+2} T \|_2^{\frac{2m}{m+2}} \\
			&\le C(n,m, \epsilon) \left\| T \right\|_{m+4}^{m+4} + \frac{\epsilon}{2}  \| \nabla^{m+2} T \|_2^{2} 
		\end{split}
	\end{equation*}
	Finally,
	\begin{equation*}
		\left\| T \right\|_{m+4}^{m+4} = \int_M |T|^{m+4} dv_g \le K^{m+2} \int_M |T|^2 dv_g.
	\end{equation*}
	No additive constant or volume factor is introduced.
\end{proof}

\begin{proposition}\label{thm:basic-energy-inequality}
	If $|\Rm| \le K$ along a solution on a closed manifold, then for every $m \ge 0$ there is $C_m < \infty$, depending only on $n, m$, such that
	\begin{equation}\label{eq:basic-energy-inequality}
		\mathcal E_m' + a_n \mathcal E_{m+2} \le C_m K^{m+2} \mathcal E_0.
	\end{equation}
\end{proposition}

\begin{proof}
	Fix $m$, and all integrals and norms in this proof are taken with respect to $g(t)$. We first differentiate the energy, 
	\begin{equation*}
		\mathcal E_m' = 2 \int_M \langle \partial_t \nabla^m \Rm, \nabla^m \Rm \rangle dv_g	+ \int_M P_{3,2m+2}(\Rm) + P_{4,2m}(\Rm)  dv_g,
	\end{equation*}
	For the rough biharmonic term, the absence of a boundary gives exactly
	\begin{equation*}
		\int_M \langle \nabla^m \Rm, \Delta^2 \nabla^m \Rm \rangle dv_g = -\int_M \langle \nabla \nabla^m \Rm,\nabla \Delta \nabla^m \Rm \rangle dv_g = \int_M | \Delta \nabla^m \Rm |^2 dv_g.
	\end{equation*}
	Substituting \Cref{eq:Rm-derivative-evolution}, we therefore obtain
	\begin{equation}\label{eq:energy-derivative}
		\begin{split}
			\mathcal E_m'
			=& -2 a_n \int_M |\Delta \nabla^m \Rm|^2 dv_g - \gamma_n \int_M \langle \nabla^m \Rm ,\nabla^m (g \odot \nabla^2 \Delta R) \rangle dv_g\\
			& + \int_M \left( P_{3,2m+2}(\Rm) + P_{4,2m}(\Rm) \right) dv_g.
		\end{split}
	\end{equation}
	where $a_n=(n-2)^{-2}$. We next compute the scalar-curvature principal term, keeping track of the commutators when $m > 0$. For any symmetric $2$-tensor $H$, the curvature symmetries and \Cref{eq:KN} give
	\begin{equation*}
		\langle \Rm,g \odot H \rangle = 4 \langle \Ric, H \rangle.
	\end{equation*}
	Since $\nabla g = 0$, the same contraction identity holds after taking
	$m$ covariant derivatives. Thus
	\begin{equation*}
		\begin{split}
			I_m := \int_M \left\langle \nabla^m \Rm, \nabla^m (g \odot \nabla^2 \Delta R) \right\rangle dv_g = 4 \int_M \left\langle \nabla^m \Ric,\nabla^m( \nabla^2 \Delta R) \right\rangle dv_g.
		\end{split}
	\end{equation*}
	In the following calculation, the divergences of $\nabla^m \Ric$ are taken in the Ricci slots, leaving the $m$ derivative slots free. The Hessian slots are paired with the Ricci slots, while the remaining derivative slots are paired in their original order. Commuting the two Hessian derivatives past the $m$ covariant derivatives produces terms of type $P_{2,m+2}(\Rm)$, we obtain
	\begin{equation*}
		\begin{split}
			I_m 
			&= 4 \int_M \left\langle \nabla^m \Ric, \nabla^2( \nabla^m \Delta R) \right\rangle dv_g + \int_M P_{3,2m+2} (\Rm) dv_g\\
			&= 4 \int_M \left\langle \divv^2 (\nabla^m \Ric), \nabla^m \Delta R \right\rangle dv_g +\int_M P_{3,2m+2}(\Rm) dv_g.
		\end{split}
	\end{equation*}
	Commuting the two divergences past the $m$ covariant derivatives and contracted Bianchi identity gives
	\begin{equation*}
		\divv^2 (\nabla^m \Ric) = \nabla^m( \divv^2 \Ric) + P_{2,m}(\Rm) = \frac{1}{2} \nabla^m \Delta R + P_{2,m}(\Rm).
	\end{equation*}
	After pairing with $\nabla^m \Delta R$, the additional term also has type $P_{3,2m+2}(\Rm)$. Consequently,
	\begin{equation*}
		I_m =2 \|\nabla^m \Delta R\|_2^2 + \int_M P_{3,2m+2}(\Rm) dv_g.
	\end{equation*}
	It follows that
	\begin{equation}\label{eq:gamma-perfect-square}
		-\gamma_n \int_M \left\langle \nabla^m \Rm, \nabla^m (g \odot \nabla^2 \Delta R)\right\rangle dv_g = -2 \gamma_n \|\nabla^m \Delta R\|_2^2 + \int_M P_{3,2m+2}(\Rm) dv_g.
	\end{equation}
	When $m = 0$, no commutations are needed. For every $m$, the principal square is nonpositive because $\gamma_n \ge 0$ for $n \ge 4$.

	We now convert the rough-Laplacian into the full $(m+2)$-nd derivative norm as follows
	\begin{equation*}
		\begin{split}
			\int_M |\nabla^2 \nabla^m \Rm|^2 dv_g = \int_M | \Delta \nabla^m \Rm|^2 dv_g + \int_M P_{3,2m+2}(\Rm) dv_g.
		\end{split}
	\end{equation*}
	Thus
	\begin{equation*}
		\|\nabla^{m+2} \Rm\|_2^2 \le \|\Delta \nabla^m \Rm \|_2^2 + \left| \int_M P_{3,2m+2}(\Rm) dv_g \right|.
	\end{equation*}
	All the terms above, including those from the changing metric and volume and the rough-Laplacian calculation, have types $P_{3,2m+2}(\Rm)$ and $P_{4,2m}(\Rm)$. Applying \Cref{lem:polynomial-interpolation} with $T = \Rm$ to these finitely many integrals, with suitably chosen parameters, gives for every $\varepsilon > 0$,
	\begin{equation*}
		\mathcal E_m' \le -(2a_n - \varepsilon) \|\nabla^{m+2} \Rm\|_2^2 - 2 \gamma_n \| \nabla^m \Delta R\|_2^2 + C_\varepsilon K^{m+2} \|\Rm\|_2^2.
	\end{equation*}
	Taking $\varepsilon = a_n$ and discarding the nonpositive scalar square proves \Cref{eq:basic-energy-inequality}.
\end{proof}

\subsection{Bernstein--Bando--Shi estimates}

We use a time-weighted energy argument of Bernstein--Bando--Shi type. The following lemma records the underlying differential inequality argument.

\begin{lemma}\label[lemma]{lem:energy-cascade}
	Fix $k \ge 0$ and $0 < \tau \le 1$, and let $H:[0, \tau] \to [0, \infty)$ be a bounded function. Suppose $E_0, E_2, \ldots, E_{2k+2}$ are nonnegative differentiable functions on $[0, \tau]$ satisfying
	\begin{equation}\label{eq:cascade-ODE}
		E_{2j}' + c_j E_{2j+2} \le b_j E_{2j} + d_j H(t), \qquad 0 \le j \le k,
	\end{equation}
	where $c_j > 0$, $b_j, d_j \ge 0$, and 
	\begin{equation*}
		\bar H := \max\left\{ E_0(0), \displaystyle\sup_{0 \le s \le \tau} H(s) \right\} < \infty.
	\end{equation*}
	Then
	\begin{equation}\label{eq:cascade-bound}
		E_{2k}(t) \le C_k  \bar{H}t^{-k}, \qquad 0< t \le \tau.
	\end{equation}
	The constant depends only on the finitely many coefficients in \Cref{eq:cascade-ODE}.
\end{lemma}

\begin{proof}
	Set $\alpha_0 = 1$ and choose the positive coefficients successively so that
	$j \alpha_j \le \frac{1}{2} c_{j-1} \alpha_{j-1}$. Define
	\begin{equation*}
		F_k(t) = \sum_{j=0}^k \alpha_j t^j E_{2j}(t).
	\end{equation*}
	Differentiating each summand and using \Cref{eq:cascade-ODE} gives
	\begin{equation*}
		\begin{split}
			F_k' \le  - \sum_{j=0}^k \alpha_j c_j t^j E_{2j+2} + \sum_{j=1}^k j \alpha_j t^{j-1} E_{2j} + \sum_{j=0}^k \alpha_j b_j t^j E_{2j} + H(t) \sum_{j=0}^k \alpha_j d_j t^j.
		\end{split}
	\end{equation*}
	Reindex the first sum, we have
	\begin{equation*}
		\sum_{j=1}^k (j \alpha_j - c_{j-1} \alpha_{j-1})t^{j-1} E_{2j} - \alpha_k c_k t^k E_{2k+2} \le 0.
	\end{equation*}
	Since $t \le 1$ and $H(t) \le \bar H$, the remaining terms satisfy
	\begin{equation*}
		F_k' \le  B_k F_k + D_k  \bar{H}, \qquad
		B_k  = \max_{0 \le j \le k} b_j, \qquad
		D_k  = \sum_{j=0}^k \alpha_j d_j.
	\end{equation*}
	Since $F_k(0) = E_0(0) \le \bar H$, Gronwall's inequality yields
	\begin{equation*}
		F_k(t) \le e^{B_kt} \left( F_k(0) + D_k \bar{H} t \right) \le e^{B_k} (1+D_k)  \bar{H}.
	\end{equation*}
	Finally $\alpha_k t^k E_{2k}(t) \le F_k(t)$ proves \Cref{eq:cascade-bound}.
\end{proof}
Now we are ready to prove the \Cref{theorem:B}.
\begin{proof}[Proof of \Cref{theorem:B}]
	If $K = 0$, the solution is flat and the assertion is immediate. Otherwise set
	\begin{equation}\label{eq:parabolic-rescaling-B}
		\hat g(s) = K g(K^{-2}s), \qquad 0 \le s \le K^2 \tau \le 1.
	\end{equation}
	Since $J_{cg} = c^{-1} J_g$, it is again a $J$-flow, with $| \Rm_{\hat g} | \le 1$. To see explicitly how the zeroth level controls all higher levels, write 
	\begin{equation*}
		\widehat{\mathcal E}_m(s) := \int_M \left| \nabla_{\hat g(s)}^m \Rm_{\hat g(s)}\right|_{\hat g(s)}^2 dv_{\hat g(s)}.
	\end{equation*}
	In the rescaled flow, \Cref{eq:basic-energy-inequality} reads
	\begin{equation*}
		\widehat{\mathcal E}_m'(s) + a_n \widehat{\mathcal E}_{m+2}(s) \le C_m \widehat{\mathcal E}_0(s).
	\end{equation*}
	Taking $m = 0$, discarding the nonnegative term and applying Gronwall's inequality, we obtain $\widehat{\mathcal E}_0(s) \le e^{C_0s} \widehat{\mathcal E}_0(0) \le e^{C_0}\widehat{\mathcal E}_0(0)$, since $s \le 1$. Consequently, for each even level,
	\begin{equation*}
		\widehat{\mathcal E}_{2j}'(s) + a_n \widehat{\mathcal E}_{2j+2}(s) \le C_{2j} e^{C_0} \widehat{\mathcal E}_0(0).
	\end{equation*}
	Use \Cref{lem:energy-cascade} with $c_j = a_n$, $b_j = 0$ and a constant source, $\sum_{j=0}^k \alpha_j s^j \widehat{\mathcal E}_{2j}(s)$ has initial value $\widehat{\mathcal E}_0(0)$ and derivative bounded above by $C_k \widehat{\mathcal E}_0(0)$, thus
	\begin{equation*}
		\widehat{\mathcal E}_{2j}(s) \le C_j \widehat{\mathcal E}_0(0) s^{-j}.
	\end{equation*}
	For odd levels, the first inequality in \Cref{eq:adjacent-Young} gives
	\begin{equation*}
		\widehat{\mathcal E}_{2j+1}(s) \le C \left(\widehat{\mathcal E}_{2j}(s) \widehat{\mathcal E}_{2j+2}(s) \right)^{\frac{1}{2}} \le C(C_j, C_{j+1}) \widehat{\mathcal E}_0(0) s^{-\frac{2j+1}{2}}.
	\end{equation*}
	By the scaling property, we have
	\begin{equation*}
		\widehat{\mathcal E}_m(K^2t) = K^{\frac{n}{2}-m-2} \mathcal E_m^g(t), 
		\qquad 
		\widehat{\mathcal E}_0(0) = K^{\frac{n}{2}-2} \mathcal E_0^g(0).
	\end{equation*}
	Substituting the rescaled estimate at $s = K^2 t$ gives
	\begin{equation*}
		\mathcal E_m^g(t) \le C_m \mathcal E_0^g(0) t^{-\frac{m}{2}}.
	\end{equation*}
	Taking square roots proves the stated estimate, with constants depending only on $n,m$.
\end{proof}

We recall the tensor Sobolev estimate needed to pass from integral to pointwise bounds.

\begin{lemma}\label[lemma]{lem:tensor-Sobolev}
	Let $(M^4,g)$ be a closed Riemannian manifold satisfying
	\begin{equation}\label{eq:normalized-Sobolev}
		\|u\|_4^2 \le C_S \left( \|\nabla u\|_2^2 + \|u\|_2^2 \right)
	\end{equation}
	for every smooth function $u$. For tensors of fixed rank, there is a constant $C$, depending only on $C_S$ and the rank, such that every smooth tensor $T$ of this rank satisfies
	\begin{equation}\label{eq:tensor-Sobolev}
		\|T\|_\infty \le C \sum_{j=0}^{3} \|\nabla^j T\|_2.
	\end{equation}
\end{lemma}

\begin{proof}
	The scalar inequality and Kato's inequality give
	\begin{equation*}
		\|T\|_4^2 \le C_S \left(\|\nabla T\|_2^2 + \|T\|_2^2 \right).
	\end{equation*}
	Applying \Cref{eq:normalized-Sobolev} to $|T|^2$, followed by H\"older's inequality and the preceding estimate for $T$ and $\nabla T$, yields
	\begin{equation*}
		\|T\|_8^4 \le C_S \left( 4 \|T\|_4^2 \|\nabla T\|_4^2 + \|T\|_4^4 \right) \le C \left(\sum_{j=0}^{2} \|\nabla^j T\|_2 \right)^4.
	\end{equation*}
	Apply this estimate to $T$ and $\nabla T$. The multiplicative Sobolev inequality of \cite[Theorem~6.1]{Streets2008}, with both finite exponents equal to eight, and Kato's inequality then give
	\begin{equation*}
		\|T\|_\infty \le C \|T\|_8^{\frac{1}{2}} \left( \|\nabla T\|_8 + \|T\|_8 \right)^{\frac{1}{2}} \le C\sum_{j=0}^{3} \|\nabla^j T\|_2.
	\end{equation*}
	The scalar inequalities apply to tensor norms by smooth approximation.
\end{proof}

\begin{corollary}\label{cor:BBS-pointwise}
	Let $g(t)$ be a smooth solution of the $J$-flow on a closed manifold $M^n$, $n \ge 4$, for $0 \le t \le \tau$, with $g(0) = g_0$. Suppose that $|\Rm|_{g(t)} \le K$ on $M \times [0, \tau]$ for some $K > 0$. Fix an integer $m \ge 0$.
	\begin{enumerate}
		\item If $n = 4$, there exists a constant $C_m$, depending only on $m$, such that
		\begin{equation}\label{eq:BBS-pointwise}
			\sup_M | \nabla^m \Rm |_{g(t)} \le C_m \left( K+t^{-\frac{1}{2}} \right)^{1 + \frac{m}{2}}, \qquad 0 < t \le \tau.
		\end{equation}
		\item For any $n > 4$, assume in addition that $\tau \le K^{-2}$ and that there exists a constant $C_S$, which may depend on $m$ but is independent of $t$, such that
		\begin{equation}
			\|T\|_{\infty, g(t)} \le C_S \sum_{i=0}^{k_0} \| \nabla^i T\|_{2, g(t)}, \qquad k_0 = \left\lfloor \frac{n}{2} \right\rfloor + 1,
		\end{equation}
		for every smooth covariant tensor $T$ of rank $m + 4$ and every $0 < t \le \min\{\tau,1\}$. Then there exists a constant $C(n,m)$, depending only on $n,m$, such that
		\begin{equation*}
			\|\nabla^m \Rm\|_{\infty,g(t)} \le C(n,m) C_S \|\Rm_{g_0}\|_{2,g_0} t^{-\frac{m+k_0}{4}}, \qquad 0 < t \le \min\{\tau, 1\}.
		\end{equation*}
	\end{enumerate}
\end{corollary}

\begin{proof}
	First suppose $n = 4$. Since $\gamma_4 = 0$, the leading curvature term in \Cref{eq:Rm-evolution} is $- \frac{1}{4} \Delta^2 \Rm$. The time change
	\begin{equation*}
		\tilde g(s) = g(4s), \qquad 0 \le s \le \frac{\tau}{4},
	\end{equation*}
	normalizes this coefficient to $-1$, while the metric evolution has the required schematic form $P_{1,2}(\Rm) + P_{2,0}(\Rm)$. Thus the rescaled solution satisfies the system in~\cite[Equation~(1.1)]{Streets2013}. Since $M$ is closed, the metrics are complete, and the curvature bound remains $K$. Consequently, \cite[Theorem~1.3]{Streets2013} gives
	\begin{equation*}
		\sup_M |\nabla^m \Rm|_{\tilde g(s)} \le C_m \left(K + s^{-\frac{1}{2}} \right)^{1+ \frac{m}{2}}, \qquad 0 < s \le \frac{\tau}{4},
	\end{equation*}
	where the curvature tensor and covariant derivatives are taken with respect to $\tilde g(s)$. Changing back to $t = 4s$ and absorbing the resulting numerical factor into $C_m$ proves \Cref{eq:BBS-pointwise}. This application requires no Sobolev bound.

	Now suppose $ n > 4$. Then $\gamma_n > 0$, and \Cref{eq:Rm-evolution} contains the additional fourth-order term $g \odot \nabla^2 \Delta R$. Under the additional restriction $\tau \le  K^{-2}$, we instead use \Cref{theorem:B}. Applying the assumed Sobolev inequality to the covariant tensor $\nabla^m \Rm$, 
	\begin{equation*}
		\begin{split}
			\|\nabla^m\Rm\|_{\infty,g(t)}
			\le C_S \sum_{i=0}^{k_0} \| \nabla^{m+i} \Rm \|_{2,g(t)}
			\le C_S \|\Rm_{g_0} \|_{2,g_0} \sum_{i=0}^{k_0} C_{m+i} t^{-\frac{m+i}{4}}
			\le C(n,m) C_S \| \Rm_{g_0} \|_{2,g_0} t^{- \frac{m+k_0}{4}}
		\end{split}
	\end{equation*}
	for $0 < t \le \min\{\tau, 1\}$, where the last inequality uses $t \le 1$.
\end{proof}

\subsection{Local energy inequalities}

We recall the commutation and weighted integration-by-parts identities used in the Bochner calculation, see \cite[Lemma~4.4 and Proposition~5.1]{Streets2008}.
\begin{lemma}\label[lemma]{lem:weighted-Bochner-identities}
	Let $(M^n,g)$ be a closed Riemannian manifold, and let $0 \le \eta \le 1$ be a smooth cutoff. For integers $j \ge 0$ and $s \ge 2$, the covariant commutator has the schematic form
	\begin{equation*}
		[\Delta, \nabla_a] \nabla^j \Rm = \Rm * \nabla^{j+1} \Rm + \nabla \Rm * \nabla^j \Rm.
	\end{equation*}
	Moreover,
	\begin{equation*}
		\int_M \nabla^b (\eta^s) \left\langle \nabla^a \nabla^j \Rm, \nabla_b \nabla_a \nabla^j \Rm \right\rangle dv_g = -\frac{1}{2} \int_M \Delta (\eta^s) |\nabla^{j+1} \Rm|^2 dv_g.
	\end{equation*}
\end{lemma}

For a fixed time-independent cutoff function $\eta$ and a fixed integer $s$, define
\begin{equation*}
	E_{j,s}^{\eta} = \int_M \eta^s |\nabla^j \Rm|^2 dv_g.
\end{equation*}
The following three interpolation lemmas are applied at a fixed time, so their constants require bounds on the cutoff for that time's metric.

\begin{lemma}\label[lemma]{lem:fixed-weight-interpolation}
	Let $T$ be a smooth tensor and let $0 \le \eta \le 1$ be a smooth compactly supported cutoff with $|\nabla \eta| + |\nabla^2 \eta| \le L$. For integers $0 \le l < k$, any $q \ge 2l$ and every $\varepsilon > 0$, one has
	\begin{equation}\label{eq:fixed-weight-ladder}
		\begin{split}
			\int_{\{\eta>0\}}\eta^q|\nabla^l T|^2 dv_g
			&\le \varepsilon \int_{\{\eta>0\}} \eta^{q+2(k-l)} |\nabla^k T|^2 dv_g +  C_\varepsilon\int_{\{\eta>0\}}\eta^{q-2l}|T|^2 dv_g\\
			&\le \varepsilon \int_{\{\eta>0\}} \eta^{q+2(k-l)} |\nabla^k T|^2 dv_g +  C_\varepsilon \int_{\{\eta>0\}}  |T|^2 dv_g.
		\end{split}
	\end{equation}
	The constant $C_\varepsilon$ depends only on $n,l,k,q,L,\varepsilon$ and the tensor rank. 
\end{lemma}
\begin{proof}
	Suppose $l \ge 1$. Integration by parts and Cauchy--Schwarz give
	\begin{equation*}
		\begin{split}
			\int_{\{\eta>0\}}\eta^q|\nabla^l T|^2 dv_g
			& = - \int_M \eta^q \langle \nabla^{l-1} T, \Delta \nabla^{l-1} T \rangle dv_g -q \int_M \eta^{q-1} \nabla_a \eta \langle \nabla^{l-1} T, \nabla^a \nabla^{l-1} T \rangle dv_g\\
			&\le \sqrt{n} \left( \int_{\{\eta>0\}} \eta^{q-2} |\nabla^{l-1} T|^2 dv_g \right)^{\frac{1}{2}} \left( \int_{\{\eta>0\}} \eta^{q+2} |\nabla^{l+1}T|^2 dv_g \right)^{\frac{1}{2}}\\
			&+ q L \left(\int_{\{\eta>0\}} \eta^{q-2} |\nabla^{l-1}T|^2 dv_g \right)^{\frac{1}{2}} \left( \int_{\{\eta>0\}} \eta^q |\nabla^l T|^2 dv_g \right)^{\frac{1}{2}}.
		\end{split}
	\end{equation*}
	Young's inequality and absorption yield, for every $\delta > 0$,
	\begin{equation*}
		\int_{\{\eta>0\}} \eta^q | \nabla^l T |^2 dv_g \le \delta \int_{\{\eta > 0\}} \eta^{q+2} | \nabla^{l+1}T |^2 dv_g + C_\delta \int_{\{\eta>0\}} \eta^{q-2} | \nabla^{l-1} T |^2 dv_g.
	\end{equation*}
	We now use induction on $l$ to remove the last derivative term. For $l = 1$ it is already of order zero. For $l > 1$, apply the inductively obtained estimate between levels $l - 1$ and $l$ with parameter $\theta > 0$:
	\begin{equation*}
		\int_{\{\eta>0\}} \eta^{q-2} | \nabla^{l-1} T |^2 dv_g \le \theta \int_{\{\eta > 0\}}\eta^q | \nabla^l T |^2 dv_g + C_\theta \int_{\{\eta>0\}} \eta^{q - 2l} |T|^2 dv_g.
	\end{equation*}
	Choose $\theta$ so that $C_\delta \theta \le \frac{1}{2}$ and absorb. This gives
	\begin{equation*}
		\int_{\{\eta > 0\}} \eta^q |\nabla^l T|^2 dv_g \le 2 \delta \int_{\{\eta >0 \}} \eta^{q+2} |\nabla^{l + 1}T|^2 dv_g + C_\delta \int_{\{\eta > 0\}} \eta^{q - 2l} |T|^2 dv_g.
	\end{equation*}
	Iterating from level $l$ to level $k$, with the weight increased by two at each step, proves the first inequality in \Cref{eq:fixed-weight-ladder}. Since $q - 2l \ge 0$ and $0 \le \eta \le 1$, the second inequality follows. The case $l = 0$ is immediate.
\end{proof}

\begin{lemma}\label[lemma]{lem:weighted-Hamilton-interpolation}
	Let $T$ be a smooth tensor and let $0 \le \eta \le 1$ be a smooth compactly supported cutoff with $|\nabla \eta| + |\nabla^2 \eta| \le L$. Suppose $|T| \le K$ on $\{\eta>0\}$. If $k \ge 1$ is an integer and $s \ge 2k$, then for every integer $1 \le i \le k$, one has
	\begin{equation}\label{eq:weighted-Hamilton-integral}
		\int_{\{\eta>0\}} \eta^s |\nabla^i T|^{\frac{2k}{i}} dv_g \le C K^{\frac{2k}{i}-2} \left( \int_{\{\eta>0\}} \eta^s |\nabla^k T|^2 dv_g + \int_{\{\eta>0\}} |T|^2 dv_g \right).
	\end{equation}
	The constant $C$ depends only on $n,k,s,L$ and the tensor rank. 
\end{lemma}
\begin{proof}
	The case $K = 0$ is immediate. For $K > 0$, the weighted interpolation inequality of Kuwert--Sch\"atzle~\cite[Theorem~5.4]{KuwertSchatzle2002}, in the local form recorded by Jachan~\cite[Lemma~B.22]{Jachan2014}, gives
	\begin{equation*}
		\begin{split}
			\left( \int_{\{\eta>0\}} \eta^s |\nabla^i T|^{\frac{2k}{i}} dv_g \right)^{\frac{i}{2k}} \le C K^{1-\frac{i}{k}} \left[ \left( \int_{\{\eta>0\}} \eta^s |\nabla^k T|^2 dv_g \right)^{\frac{1}{2}} + \left( \int_{\{\eta>0\}} |T|^2 dv_g \right)^{\frac{1}{2}} \right]^{\frac{i}{k}},
		\end{split}
	\end{equation*}
	with the cutoff dependence absorbed into $C$. Raising both sides to the power $\frac{2k}{i}$ and using $(a+b)^2 \le 2 a^2 + 2 b^2$ proves \Cref{eq:weighted-Hamilton-integral}, since $(1-i/k)(2k/i) = 2k/i-2$.
\end{proof}

\begin{lemma}\label[lemma]{lem:fixed-weight-product}
	Let $T$ be a smooth tensor and let $0 \le \eta \le 1$ be a smooth compactly supported cutoff with $|\nabla \eta| + |\nabla^2 \eta| \le L$. If $k \ge 1$ and $r \ge 2$ are integers, $|T| \le K$ on $\{\eta > 0\}$, $s \ge 2k$, and the integers $0 \le i_1, \ldots , i_r \le k$ satisfy $i_1 + \cdots + i_r =2 k$, then
	\begin{equation}\label{eq:fixed-weight-product}
		\left| \int_M \eta^s \nabla^{i_1} T * \cdots * \nabla^{i_r} T dv_g \right| \le C K^{r-2} \left( \int_{\{\eta>0\}} \eta^s |\nabla^k T|^2 dv_g + \int_{ \{\eta>0\}} |T|^2 dv_g \right).
	\end{equation}
	The constant $C$ depends only on $n,k,r,s,L$, the tensor rank and the contraction coefficients. It does not depend on a Sobolev constant.
\end{lemma}
\begin{proof}
	If $K = 0$, the assertion is immediate. Assume $K > 0$ and bound the undifferentiated factors by $K$. H\"older's inequality and \Cref{lem:weighted-Hamilton-interpolation} yield
	\begin{equation*}
		\begin{split}
			\left| \int_M \eta^s \nabla^{i_1} T * \cdots * \nabla^{i_r} T dv_g \right|
			&\le C\left( \prod_{\{\nu:i_\nu=0\}} K \right) \prod_{\{\nu:i_\nu>0\}}
			\left( \int_{\{\eta>0\}} \eta^s |\nabla^{i_\nu}T|^{\frac{2k}{i_\nu}} dv_g \right)^{\frac{i_\nu}{2k}}\\
			&\le CK^{\sum_{\nu=1}^r (1- \frac{i_\nu}{k})} \left( \int_{\{ \eta>0\}} \eta^s |\nabla^k T|^2 dv_g +\int_{\{\eta>0\}} |T|^2 dv_g \right)^{\sum_{\nu=1}^r \frac{i_\nu}{2k}}\\
			&=C K^{r-2} \left( \int_{\{\eta>0\}} \eta^s |\nabla^k T|^2 dv_g + \int_{\{\eta>0\}} |T|^2 dv_g \right).
		\end{split}
	\end{equation*}
	This proves \Cref{eq:fixed-weight-product}.
\end{proof}

\begin{proposition}\label{thm:local-energy}
	Let $g(t)$ be a smooth $J$-flow in dimension $n \ge 4$. Suppose $| \Rm | \le K$ on $\{\eta > 0\} \times [0, \tau]$, where $\eta$ is smooth, compactly supported and time independent, $0 \le \eta \le 1$, and $|\nabla \eta| + |\nabla^2 \eta| \le L$ throughout this interval. For $s \ge 2j + 4$,
	\begin{equation}\label{eq:local-energy-inequality}
		\frac{d}{dt} E_{j,s}^{\eta} + c_j E_{j+2,s}^{\eta} \le C(n,j,s,L) K^2 E_{j,s}^{\eta} + C(n,j,s,L)(1 + K^2) H_\eta(t),
	\end{equation}
	where $c_j > 0$ depends only on $n, j$ and $H_\eta(t) = \int_{\{\eta>0\}} |\Rm|^2 dv_g$.
\end{proposition}

\begin{proof}
	All integrals and covariant derivatives are taken with respect to $g(t)$, and the constants depend on $n, j, s, L$ unless otherwise indicated. Since $\eta$ is time independent, differentiating the energy and then using \Cref{eq:Rm-derivative-evolution} gives
	\begin{equation*}
		\begin{split}
			\frac{d}{dt} E_{j,s}^{\eta}
			&= -2a_n \int_M \eta^s \left\langle \nabla^j \Rm, \Delta^2 \nabla^j \Rm \right\rangle dv_g - \gamma_n \int_M \eta^s \left\langle \nabla^j \Rm, \nabla^j(g \odot \nabla^2 \Delta R) \right\rangle dv_g\\
			&\quad + \int_M \eta^s \left( P_{3,2j+2}(\Rm) + P_{4,2j}(\Rm) \right) dv_g.
		\end{split}
	\end{equation*}
	We first treat the biharmonic term. Integration by parts and \Cref{lem:weighted-Bochner-identities} gives
	\begin{equation*}
		\begin{split}
			&-2 a_n \int_M \eta^s \left\langle \nabla^j \Rm,\Delta^2 \nabla^j \Rm \right\rangle dv_g\\
			=& -2 a_n \int_M \eta^s |\Delta \nabla^j \Rm|^2 dv_g - 4 a_n \int_M \nabla^a(\eta^s) \left\langle \nabla_a \nabla^j \Rm, \Delta \nabla^j \Rm \right\rangle dv_g\\
			&- 2 a_n \int_M \Delta(\eta^s) \left\langle \nabla^j \Rm, \Delta \nabla^j \Rm \right\rangle dv_g\\
			=&-2 a_n E_{j+2,s}^{\eta} + a_n \int_M \Delta(\eta^s) |\nabla^{j+1} \Rm|^2 dv_g - 2a_n \int_M \nabla^a(\eta^s) \left\langle \nabla_a \nabla^j \Rm,\Delta \nabla^j \Rm \right\rangle dv_g \\
			&-2 a_n \int_M \Delta(\eta^s) \left\langle \nabla^j \Rm,\Delta \nabla^j \Rm \right\rangle dv_g - 2a_n \int_M \eta^s \left\langle \nabla^a \nabla^j \Rm,[\Delta, \nabla_a] \nabla^j \Rm \right\rangle dv_g\\
			\le& - (2a_n - \delta) E_{j+2,s}^{\eta} + \left| \int_M \eta^s P_{3,2j+2}(\Rm)dv_g \right| + C_\delta \int_{\{\eta>0\}} \eta^{s-2} |\nabla^{j+1} \Rm|^2 dv_g + C_\delta \int_{\{\eta>0\}} \eta^{s-4} |\nabla^j\Rm|^2 dv_g\\
			\le& - (2a_n - \delta) E_{j+2,s}^{\eta} + C_\delta\theta E_{j+2,s}^{\eta} + \left| \int_M \eta^s P_{3,2j+2}(\Rm)dv_g \right| + C_\delta C_\theta H_\eta(t)\\
			\le& - (2a_n - \varepsilon) E_{j+2,s}^{\eta} + \left| \int_M \eta^s P_{3,2j+2}(\Rm)dv_g \right| + C_\varepsilon H_\eta(t).
		\end{split}
	\end{equation*}
	The first inequality follows from $|\nabla(\eta^s)|\le C\eta^{s-1}$, $|\nabla^2 (\eta^s)| \le C \eta^{s-2}$ and Young's inequality. For any prescribed $\varepsilon > 0$, first choose $0 < \delta \le \varepsilon / 2$ and then apply \Cref{eq:fixed-weight-ladder} to the two lower-weight integrals, from levels $j+1$ and $j$ to $j+2$, with weights $s-2$ and $s-4$, respectively, choosing the interpolation parameters so that the sum of these integrals is bounded by $\theta E_{j+2,s}^{\eta} + C_\theta H_\eta(t)$. Since $s \ge 2j + 4$, the resulting zeroth-order weights are nonnegative, when $j = 0$, the second integral is bounded directly by $H_\eta(t)$. Finally, choose $\theta > 0$ so that $C_\delta\theta\le\varepsilon/2$. Then $\delta+C_\delta \theta \le \varepsilon$, and absorbing $C_\delta C_\theta$ into $C_\varepsilon$ gives the last inequality. The remaining cubic integral is estimated below together with the other curvature polynomial terms.

	We next estimate the curvature polynomial terms. The weighted interpolation estimates \cite[Lemmas~10.1--10.3]{Streets2008}, together with \Cref{eq:fixed-weight-ladder,eq:fixed-weight-product}, give, for every $\varepsilon > 0$,
	\begin{equation}\label{eq:weighted-polynomial-bound}
		\begin{split}
			\left| \int_M \eta^s P_{3, 2j + 2} (\Rm) dv_g \right| + \left| \int_M \eta^s P_{4,2 j}(\Rm) dv_g \right| \le \varepsilon E_{j + 2,s}^{\eta} + C_\varepsilon K^2 E_{j,s}^{\eta} + C_\varepsilon(1 + K^2) H_\eta(t).
		\end{split}
	\end{equation}
	This estimate also covers a cubic term with derivative orders $j + 2, a, b$, where $a+b=j$, Cauchy--Schwarz and \Cref{eq:fixed-weight-product} with $k = j$, $r = 4$ applied to the squared product of the last two factors give
	\begin{equation*}
		\int_M \eta^s |\nabla^{j+2} \Rm| |\nabla^a \Rm| |\nabla^b \Rm| dv_g \le C K \left( E_{j+2,s}^{\eta} \right)^{\frac{1}{2}} \left( E_{j,s}^{\eta} + H_\eta(t) \right)^{\frac{1}{2}},
	\end{equation*}
	and Young's inequality applies. When $j = 0$, the same bound follows directly from $|\Rm| \le K$.

	Applying these bounds to the differentiated energy and choosing the Young parameters small enough, as in \cite[Corollary~5.2]{Streets2008}, gives
	\begin{equation}\label{eq:local-energy-before-scalar}
		\frac{d}{dt} E_{j,s}^{\eta} + \frac{a_n}{2} E_{j+2,s}^{\eta} \le -\gamma_n \int_M \eta^s \left\langle \nabla^j \Rm, \nabla^j (g \odot \nabla^2 \Delta R) \right\rangle dv_g + C K^2 E_{j,s}^{\eta} + C(1+K^2) H_\eta(t).
	\end{equation}

	It remains to estimate the scalar-curvature term. As in the proof of \Cref{eq:gamma-perfect-square}, commuting the Hessian derivatives past $\nabla^j$ and integrating by parts twice gives
	\begin{equation*}
		\int_M \eta^s \left\langle \nabla^j \Rm, \nabla^j (g \odot \nabla^2 \Delta R) \right\rangle dv_g = 4 \int_M \left\langle \divv^2 (\eta^s \nabla^j \Ric), \nabla^j \Delta R \right\rangle dv_g + \int_M \eta^s P_{3,2j+2} (\Rm) dv_g.
	\end{equation*}
	The product rule and the contracted Bianchi identity yield
	\begin{equation*}
		\divv^2 (\eta^s \nabla^j \Ric) = \frac{1}{2} \eta^s \nabla^j \Delta R + \eta^s P_{2,j}(\Rm) + \nabla(\eta^s) * \nabla^{j+1} \Ric + \nabla^2 (\eta^s) * \nabla^j \Ric.
	\end{equation*}
	The curvature commutators vanish when $j = 0$. In general, after pairing with $\nabla^j \Delta R$, all cubic terms have total derivative order $2j + 2$ and no factor has more than $j + 2$ derivatives, so \Cref{eq:weighted-polynomial-bound} applies. For the remaining terms, $|\nabla(\eta^s)| \le C \eta^{s-1}$, $|\nabla^2 (\eta^s)| \le C \eta^{s-2}$ and $|\nabla^j \Delta R| \le C(n) |\nabla^{j+2}\Rm|$ give
	\begin{equation*}
		\begin{split}
			&C \int_{\{\eta>0\}} \eta^{s-1} |\nabla^{j+1} \Rm| |\nabla^{j+2} \Rm| dv_g + C \int_{\{\eta>0\}} \eta^{s-2} |\nabla^j \Rm| |\nabla^{j+2} \Rm| dv_g\\
			\le& \delta E_{j+2,s}^{\eta} + C_\delta \int_{\{\eta>0\}} \eta^{s-2} |\nabla^{j+1} \Rm|^2 dv_g + C_\delta \int_{\{\eta>0\}} \eta^{s-4} |\nabla^j \Rm|^2 dv_g\\
			\le& \varepsilon E_{j+2,s}^{\eta}  + C_\varepsilon H_\eta(t).
		\end{split}
	\end{equation*}
	Here we first choose $\delta$ and then apply \Cref{eq:fixed-weight-ladder} from levels $j+1$ and $j$ to $j+2$, with weights $s-2$ and $s-4$, respectively. When $j = 0$, the second lower-weight integral is bounded directly by $H_\eta(t)$. Since $\gamma_n \ge 0$, it follows that
	\begin{equation}\label{eq:weighted-scalar-bound}
		-\gamma_n \int_M \eta^s \left\langle \nabla^j \Rm, \nabla^j(g \odot \nabla^2 \Delta R) \right\rangle dv_g \le - 2 \gamma_n \int_M \eta^s |\nabla^j \Delta R|^2 dv_g + \varepsilon E_{j+2,s}^{\eta} + C_\varepsilon K^2 E_{j,s}^{\eta} + C_\varepsilon(1+K^2) H_\eta(t).
	\end{equation}
	Combining \Cref{eq:local-energy-before-scalar,eq:weighted-scalar-bound}, taking $\varepsilon = a_n / 4$ and discarding the nonpositive terms proves \Cref{eq:local-energy-inequality} with $c_j = a_n / 4$. In dimension four the calculation is simpler because $\gamma_4 = 0$, so the scalar-curvature principal term is absent.
\end{proof}

\begin{lemma}\label[lemma]{lem:weighted-interpolation}
	Let $T$ be a smooth tensor and let $0 \le \phi \le 1$ be a smooth compactly supported cutoff with $|\nabla \phi| \le L$. Then
	\begin{equation}\label{eq:interpolation-recursion}
		\int_M \phi^2 |\nabla T|^2 dv_g \le C(n) \|T\|_{L^2(\{\phi > 0\})} \|\nabla^2T\|_{L^2(\{\phi>0\})} + C(n) L^2 \|T\|_{L^2(\{\phi > 0\})}^2.
	\end{equation}
\end{lemma}

\begin{proof}
	Compact support permits the integration by parts
	\begin{equation*}
		\int_M \phi^2 |\nabla T|^2 dv_g = - \int_M \phi^2 \langle T, \Delta T \rangle d v_g - 2 \int_M \phi \langle T, \nabla_{\nabla \phi} T \rangle dv_g.
	\end{equation*}
	Using $|\Delta T| \le \sqrt{n} |\nabla^2T|$ in the first integral and Young's inequality in the second gives
	\begin{equation*}
		\int_M\phi^2|\nabla T|^2 dv_g \le \sqrt{n} \|T\|_{L^2(\{\phi>0\})} \|\nabla^2 T \|_{L^2(\{\phi>0\})} + \frac{1}{2} \int_M \phi^2 |\nabla T|^2 dv_g + 2 L^2 \|T\|_{L^2(\{\phi>0\})}^2.
	\end{equation*}
	Moving the middle term to the left proves the assertion. The computation uses no commutation of covariant derivatives.
\end{proof}

\begin{proof}[Proof of \Cref{theorem:C}]
	We apply the same time-weighted energy argument to the localized energies.
	Fix $t_* \le \min\{\tau, 1\}$ and set
	\begin{equation*}
		H_*: = \sup_{0 \le t \le t_*} \int_{B_{2r}(p)} |\Rm|^2 dv_{g(t)}.
	\end{equation*}
	Choose the cutoff $\Phi$ from \Cref{eq:cutoff-estimate}. To estimate all even levels up to $2k$, choose a single integer $s \ge 4k+4$. \Cref{thm:local-energy} gives
	\begin{equation*}
		\frac{d}{dt} E_{2j,s}^{\Phi}(t) + c_{2j} E_{2j+2,s}^{\Phi}(t) \le b_j E_{2j,s}^{\Phi}(t) + d_j H_*, \qquad 0 \le j \le k,
	\end{equation*}
	where $b_j = C(n,2j,s,L) K^2$ and $d_j = C(n,2j,s,L) (1+K^2)$. Here we used $H_\Phi(t)\le H_*$ for $0 \le t \le t_*$, since $\{\Phi>0\} \subset B_{2r}(p)$. The equality of the weights at adjacent levels is essential in the following calculation. Set 
	\begin{equation*}
		G_k(t) = \sum_{j=0}^k \alpha_j t^j E_{2j,s}^{\Phi}(t),
	\end{equation*}
	where $\alpha_0 = 1$ and $j \alpha_j \le \frac{1}{2}c_{2j-2} \alpha_{j-1}$. Differentiating and reindexing, we have
	\begin{equation*}
		\begin{split}
			G_k'(t)
			\le& \sum_{j=1}^k(j \alpha_j - c_{2j-2} \alpha_{j-1}) t^{j-1} E_{2j,s}^{\Phi}(t) - c_{2k} \alpha_k t^k E_{2k+2,s}^{\Phi}(t) + \sum_{j=0}^k \alpha_j b_j t^j E_{2j,s}^{\Phi}(t) + H_* \sum_{j=0}^k \alpha_j d_j t^j\\
			\le& C G_k(t) + C H_*.
		\end{split}
	\end{equation*}
	Since $G_k(0) = E_{0,s}^{\Phi}(0) \le H_*$ and $t_* \le1$, integration gives $G_k (t_*) \le C(n,k,K,L) H_*$. In particular,
	\begin{equation*}
		E_{2k,s}^{\Phi}(t_*) \le C(n,k,K,L) H_* t_*^{-k}.
	\end{equation*}
	Because $\Phi = 1$ on $\{\phi > 0\}$, this bounds the unweighted even-level energy on that entire set. 

	For $m = 2k+1$, obtain the even estimates at levels $2k$ and $2k + 2$ by choosing $s \ge 4k + 8$. Apply \Cref{eq:interpolation-recursion} to $T = \nabla^{2k} \Rm$. Since $\phi = 1$ on $B_r(p)$,
	\begin{equation*}
		\int_{B_r(p)} |\nabla^{2k+1} \Rm|^2 dv_{g(t_*)} \le C H_* \left(t_*^{-\frac{2k+1}{2}} + L^2 t_*^{-k} \right) \le C(n,k,K,L) H_* t_*^{ -\frac{2k+1}{2} }.
	\end{equation*}
	This proves \Cref{eq:local-BBS} for every $m$, since $t_*$ was arbitrary.
\end{proof}

\begin{remark}\label{rmk:local-higher-dim}
	For the scaled version, assume $r |\nabla f| + r^2 |\nabla^2f| \le L$ for both cutoffs and use $\hat g(s) = r^{-2} g(r^4 s)$. Then \Cref{eq:local-BBS} holds for $0 < t \le \min\{\tau, r^4\}$ with $C = C(n,m,K r^2,L)$. 
\end{remark}

\section{Long-time existence under Sobolev control}\label{sec:long-time}\label{subsec:continuation}

Throughout this section, the manifold is closed and four-dimensional. We prove \Cref{theorem:D} by combining the global integral estimates of \Cref{theorem:B} with the assumed Sobolev inequality to rule out unbounded curvature. We first recall the continuation criterion for bounded curvature, whose proof uses the four-dimensional pointwise smoothing estimate.

\begin{proposition}\label{prop:curvature-continuation}
	If a smooth $J$-flow on $[0 , T )$, $T < \infty$, has bounded curvature, it extends smoothly past $T$. In particular a finite maximal time implies $\limsup_{t \rightarrow T}\|\Rm_{g(t)}\|_\infty = \infty$.
\end{proposition}

\begin{proof}
	Fix $t_1 \in (0,T)$. By \Cref{cor:BBS-pointwise}, all covariant curvature derivatives are uniformly bounded for $t \in [t_1,T)$. Since $h = \partial_t g = -2J$ is a universal sum of $\nabla^2 \Rm * g$ and $\Rm * \Rm$, all $\nabla^j h$ are bounded there, in particular $|h|_g \le C$. Integrating $|\partial_t \log g(v,v)| \le C$ gives
	\begin{equation*}
		e^{-C(T - t_1)} g(t_1) \le g(t) \le e^{C(T - t_1)} g(t_1).
	\end{equation*}
	Thus the metrics remain positive and uniformly equivalent on this finite interval.

	These intrinsic estimates also control the metrics in fixed coordinates. The connection variation formula gives
	\begin{equation*}
		\partial_t \left(\nabla^{g(t)} - \nabla^{g(t_1)} \right)^k_{ij} = \frac{1}{2} g^{kl} (\nabla_i h_{jl} + \nabla_j h_{il} - \nabla_l h_{ij}).
	\end{equation*}
	The connection difference vanishes at $t = t_1$, so integration, metric equivalence and the bound on $\nabla h$ give a uniform bound for $\nabla^{g(t)} - \nabla^{g(t_1)}$ in the fixed metric $g(t_1)$. Since $\nabla^{g(t)} g(t) =0$, the identity
	\begin{equation*}
		\nabla^{g(t_1)}g(t) = \left(\nabla^{g(t)} - \nabla^{g(t_1)} \right) * g(t)
	\end{equation*}
	then bounds the first derivatives of $g(t)$ with respect to the fixed connection. 
	
	At each higher order, apply $(\nabla^{g(t_1)})^j$ to the connection variation formula and express the fixed-connection derivatives of $h$ using $\nabla^{g(t)}$ and the connection difference. The resulting equation involves bounded covariant derivatives of $h$, lower fixed-connection derivatives of $\nabla^{g(t)} - \nabla^{g(t_1)}$, and terms at most linear in $(\nabla^{g(t_1)})^j (\nabla^{g(t)} - \nabla^{g(t_1)})$. Gronwall's inequality on the finite interval and the preceding identity therefore bound all fixed-connection derivatives of the connection difference and of $g(t)$. Consequently $(\nabla^{g(t_1)})^j h$ is uniformly bounded for every $j$.

	It follows from $g(t) - g(s) = \int_s^t h(u) du$ that $g(t)$ is Cauchy in every fixed $C^j$ norm as $t \rightarrow T$. Its limit $g(T)$ is smooth and positive definite. Apply \Cref{theorem:A} to $g(T)$ and join the solutions, the equation supplies matching time derivatives, and uniqueness gives a smooth extension.
\end{proof}

Combining the tensor Sobolev estimate in \Cref{lem:tensor-Sobolev} with \Cref{theorem:B} gives the following estimate at the curvature scale.

\begin{lemma}\label[lemma]{lem:integral-pointwise}
	Let $g(t)$ be a smooth $J$-flow on a closed four-manifold for $0 \le t \le1$, with $|\Rm| \le 1$ throughout this interval. If \Cref{eq:normalized-Sobolev} holds with respect to $g(1)$, then
	\begin{equation}\label{eq:integral-pointwise}
		\|\Rm\|_{\infty, g(1)} \le C(C_S) \|\Rm\|_{2,g(0)}.
	\end{equation}
\end{lemma}

\begin{proof}
	Apply \Cref{eq:tensor-Sobolev} to the curvature tensor at time one and \Cref{theorem:B} with $K = \tau = 1$ at derivative orders $0,1,2,3$. This gives
	\begin{equation*}
		\| \Rm \|_{\infty,g(1)} \le C(C_S) \sum_{j=0}^{3} \|\nabla^j \Rm \|_{2,g(1)} \le C(C_S) \| \Rm \|_{2,g(0)}.
	\end{equation*}
\end{proof}

\begin{proof}[Proof of \Cref{theorem:D}]
	Let $C(C_S)$ be the constant in \Cref{eq:integral-pointwise} and choose $\varepsilon(C_S) = (4C(C_S)^2)^{-1}$. Suppose that curvature is unbounded on the maximal interval, whether that interval is finite or infinite. Choose $K_i \to \infty$, greater than the initial curvature maximum, and let $t_i$ be the first time at which $\|\Rm\|_{\infty, g(t_i)} = K_i$. The first-time choice gives $|\Rm| \le K_i$ on $[0,t_i]$. Since the solution is smooth on an initial interval, the $t_i$ are bounded away from zero for large $i$. Thus $K_i^2 t_i \ge 1$ and $K_i r_0^2 \ge 1$ for all sufficiently large $i$.

	Rescale the final interval of length $K_i^{-2}$ by setting
	\begin{equation*}
		\hat g_i(s) = K_i g \left(t_i + K_i^{-2}(s-1) \right), \qquad 0 \le s \le 1.
	\end{equation*}
	After scaling, this is again a solution, with $|\Rm_{\hat g_i(s)}| \le 1$ throughout and $\|\Rm\|_{\infty, \hat g_i(1)} = 1$. At its final time, \Cref{eq:uniform-Sobolev-D} becomes
	\begin{equation*}
		\|u\|_{4, \hat g_i(1)}^2
		\le C_S \left( \|\nabla u\|_{2, \hat g_i(1)}^2 + (K_i r_0^2)^{-1} \|u\|_{2, \hat g_i(1)}^2 \right)\\
		\le C_S \left( \|\nabla u\|_{2, \hat g_i(1)}^2 + \|u\|_{2, \hat g_i(1)}^2 \right).
	\end{equation*}
	The four-dimensional curvature energy is unchanged by constant metric rescaling, so the assumed energy bound at the beginning of this interval gives
	\begin{equation*}
		\|\Rm\|_{2, \hat g_i(0)}^2 = \|\Rm\|_{2,g(t_i - K_i^{-2})}^2 < \varepsilon(C_S).
	\end{equation*}
	Applying \Cref{eq:integral-pointwise} to $\hat g_i$ now yields
	\begin{equation*}
		1 = \|\Rm\|_{\infty, \hat g_i(1)} \le C(C_S) \|\Rm\|_{2,\hat g_i(0)} < C(C_S) \sqrt{\varepsilon(C_S)} = \frac{1}{2},
	\end{equation*}
	a contradiction. Curvature is therefore uniformly bounded on the whole maximal interval. If $T_{\max} < \infty$, \Cref{prop:curvature-continuation} extends the solution past $T_{\max}$, a contradiction.
\end{proof}

\begin{remark}\label{rmk:continuation-scope}
	The proof uses the assumed curvature-energy bound at the earlier times $t_i - K_i^{-2}$ and the Sobolev inequality at the first maximum times $t_i$. The constant in \Cref{eq:integral-pointwise} depends only on $C_S$, so the threshold is independent of $g_0$ and $r_0$. 
\end{remark}

\section{Estimates for the J-tensor and smooth convergence}\label{sec:J-estimates}

This section proves \Cref{theorem:E}. We estimate the tensor driving the evolution, using $\partial_t g = -2 J_g$ to relate its time integral to convergence of the metrics. In dimension four, the trace and divergence identities cancel the mixed fourth-order terms in the evolution of $J$. We use this cancellation to prove smoothing estimates that retain its initial $L^2$-norm, and then derive smooth convergence from time integrability of the $L^2$-norm of $J$. Throughout this section, covariant derivatives and unspecified norms are taken with respect to $g(t)$.

\begin{lemma}\label[lemma]{lem:J-evolution}
	Along a smooth $J$-flow on a closed four-manifold, the $J$-tensor satisfies
	\begin{equation}\label{eq:J-tensor-evolution}
		\partial_t J = - \frac{1}{4} \Delta^2 J + \Rm * \nabla^2 J + \nabla \Rm * \nabla J + \nabla^2 \Rm * J + \Rm * \Rm * J.
	\end{equation}
	Each schematic term denotes a finite sum of universal contractions.
\end{lemma}

\begin{proof}
	By \Cref{prop:trace-J,prop:J-Schur},
	\begin{equation*}
		\tr_g J = Q, \qquad \delta_g J = - \frac{1}{4} dQ, \qquad \delta_g^2 J = \frac{1}{4} \Delta Q.
	\end{equation*}
	Substituting $\partial_t g = -2 J$ into \Cref{eq:Ric-linearization,eq:R-linearization} therefore gives
	\begin{equation*}
		\partial_t \Ric = \Delta J + \frac{1}{2} \nabla^2 Q + \Rm * J,
		\qquad
		\partial_t R = \frac{3}{2} \Delta Q + 2 \langle \Ric, J \rangle.
	\end{equation*}
	Differentiate \Cref{eq:J-dim4-full}. The terms in which the time derivative falls on $\Ric$ or $R$ inside their second derivatives have fourth-order part
	\begin{equation*}
		\frac{1}{12} \nabla^2 \left( \frac{3}{2} \Delta Q \right) -\frac{1}{4} \Delta\left( \Delta J + \frac{1}{2} \nabla^2 Q \right) = -\frac{1}{4} \Delta^2 J + \frac{1}{8}(\nabla^2 \Delta Q - \Delta \nabla^2 Q).
	\end{equation*}
	The commutator satisfies
	\begin{equation*}
		[\Delta, \nabla^2] Q = \Rm * \nabla^2 Q + \nabla \Rm * \nabla Q.
	\end{equation*}
	Since $\nabla g = 0$ and $Q = \tr_g J$, these are contractions of $\Rm * \nabla^2 J$ and $\nabla \Rm * \nabla J$. The variations of the differential operators contribute
	\begin{equation*}
		(\partial_t \Delta) \Ric = J * \nabla^2 \Rm + \nabla J * \nabla \Rm + \nabla^2 J * \Rm, \qquad (\partial_t \nabla^2) R = \nabla J * \nabla R.
	\end{equation*}
	Finally, the variation of each quadratic curvature term is of type $\Rm * \nabla^2 J+ \Rm * \Rm * J$, and differentiating the lower-order terms in $\partial_t \Ric$ and $\partial_t R$ gives the remaining contractions in \Cref{eq:J-tensor-evolution}. In particular, every term contains $J$ or one of its derivatives.
\end{proof}

\begin{lemma}\label[lemma]{lem:J-energy}
	Let $(M^4, g_0)$ be a closed Riemannian four-manifold, and let $g(t)$ be a smooth solution of the $J$-flow on $M^4 \times [0,\infty)$ with $g(0) = g_0$. Suppose $|\Rm| \le K$. For every $t_0 > 0$ and integer $m \ge 0$, there is $C_m = C(m,K,t_0)$ such that, for $t\ge t_0$,
	\begin{equation}\label{eq:J-energy}
		\frac d{dt} \| \nabla^m J \|_2^2 + \frac{1}{8} \| \nabla^{m+2} J \|_2^2 \le C_m \|J\|_2^2.
	\end{equation}
\end{lemma}

\begin{proof}
	By \Cref{cor:BBS-pointwise} and \Cref{eq:J-dim4-full}, all covariant derivatives of $\Rm$ and $J$ are uniformly bounded for $t \ge t_0$. Commuting $\nabla^m$ through \Cref{eq:J-tensor-evolution} produces terms with at most $m + 2$ derivatives of $J$ and uniformly bounded curvature coefficients. The time commutator has the form
	\begin{equation*}
		[\partial_t, \nabla^m] J = \sum_{i=0}^{m-1} \nabla^{i+1} J * \nabla^{m-1-i} J,
	\end{equation*}
	with an empty sum when $m = 0$. Bound the first factor of each summand using the preceding uniform estimates. It follows that
	\begin{equation*}
		\left| \partial_t \nabla^m J + \frac{1}{4} \Delta^2 \nabla^m J \right| \le C_m \sum_{\ell = 0}^{m+2} |\nabla^\ell J|.
	\end{equation*}
	The changing contractions and volume element contribute at most $C_m \|\nabla^m J\|_2^2$. Differentiation and integration by parts thus give
	\begin{equation*}
		\frac{d}{dt} \| \nabla^m J \|_2^2 + \frac{1}{2} \| \Delta \nabla^m J\|_2^2 \le C_m\|\nabla^m J \|_2 \sum_{\ell=0}^{m+2} \| \nabla^\ell J \|_2.
	\end{equation*}
	The unweighted Bochner calculation, using the uniform bounds on $\Rm$ and $\nabla \Rm$, yields
	\begin{equation*}
		\|\Delta \nabla^m J\|_2^2 \ge \frac{1}{2} \|\nabla^{m+2} J \|_2^2 - C_m \left( \|\nabla^{m+1} J \|_2^2 + \|\nabla^m J \|_2^2 \right).
	\end{equation*}
	Apply Young's inequality to the highest-order factor on the right of the energy estimate. All remaining positive derivative levels are controlled by covariant $L^2$ interpolation:
	\begin{equation*}
		\sum_{\ell=1}^{m+1} \|\nabla^\ell J \|_2^2 \le \varepsilon\| \nabla^{m+2} J \|_2^2 + C_{m, \varepsilon} \| J \|_2^2.
	\end{equation*}
	This follows by iterating integration by parts, starting from $\|\nabla^\ell J\|_2^2 \le \sqrt{n} \|\nabla^{\ell-1} J \|_2 \|\nabla^{\ell+1} J \|_2$. Choosing the Young and interpolation parameters sufficiently small proves \Cref{eq:J-energy}. Every other term retains a factor of a derivative of $J$ in each of its two $L^2$ factors, so no additive constant or volume term is introduced.
\end{proof}

\begin{proposition}\label{prop:J-smoothing}
	Let $(M^4, g_0)$ be a closed Riemannian four-manifold, and let $g(t)$ be a smooth solution of the $J$-flow on $M^4 \times [0, \infty)$ with $g(0) = g_0$. Suppose $|\Rm| \le K$ and that \Cref{eq:uniform-Sobolev-D} holds with time-independent constants $C_S,r_0$. For every $t_0 > 0$ and integer $m \ge 0$, there is $C_m = C(m,K,t_0,C_S,r_0)$ such that, for every $t \ge t_0$ and $0 < \sigma \le 1$, we have
	\begin{equation}\label{eq:J-smoothing}
		\|\nabla^m J \|_{2,g(t + \sigma)} \le C_m \sigma^{-\frac{m}{4}} \| J \|_{2,g(t)}, 
	\end{equation}
	and
	\begin{equation*}
		\| \nabla^m J \|_{\infty,g(t+1)} \le C_m \| J \|_{2,g(t)}.
	\end{equation*}
\end{proposition}

\begin{proof}
	Fix $t \ge t_0$. The case $m = 0$ of \Cref{eq:J-energy} and Gronwall's inequality give
	\begin{equation*}
		\|J\|_{2, g(t+\sigma)}^2 \le e^{C_0 \sigma} \| J \|_{2,g(t)}^2, \qquad 0 \le \sigma \le 1.
	\end{equation*}
	Apply \Cref{lem:energy-cascade} on this shifted time interval to the energies $\| \nabla^{2j} J \|_{2, g(t + \sigma)}^2$, with coefficient $1/8$ and source bounded by $e^{C_0} \|J\|_{2,g(t)}^2$. This gives
	\begin{equation*}
		\|\nabla^{2j} J \|_{2,g(t+\sigma)}^2 \le C_j\sigma^{-j} \| J \|_{2,g(t)}^2.
	\end{equation*}
	For odd orders, integration by parts gives
	\begin{equation*}
		\|\nabla^{2j+1} J\|_2^2 \le \sqrt{n} \|\nabla^{2j} J\|_2 \| \nabla^{2j+2} J \|_2.
	\end{equation*}
	Substitution proves the first estimate in \Cref{eq:J-smoothing} for every $m$. The scalar Sobolev hypothesis implies the normalized inequality in \Cref{lem:tensor-Sobolev} with a constant depending only on $C_S, r_0$. Apply that lemma to $\nabla^m J$ at time $t + 1$ and use the first estimate at orders $m, m+1, m+2, m+3$ to obtain
	\begin{equation*}
		\|\nabla^m J \|_{\infty, g(t+1)} \le C_m \sum_{\ell=0}^3 \|\nabla^{m + \ell} J \|_{2,g(t+1)} \le C_m \|J\|_{2,g(t)}.
	\end{equation*}
	All constants are uniform in $t\ge t_0$.
\end{proof}

\begin{proof}[Proof of \Cref{theorem:E}]
	Fix $t_0>0$. The curvature bound and uniform Sobolev inequality allow us to apply \Cref{prop:J-smoothing} for $t\ge t_0$. Integrating the second estimate in \Cref{eq:J-smoothing} and using \Cref{eq:J-time-integrability} yields, for every $m\ge0$,
	\begin{equation}\label{eq:J-integrable-derivatives}
		\int_{t_0+1}^\infty \|\nabla^m J \|_{\infty,g(t)} dt \le C_m \int_{t_0}^\infty \|J\|_{2,g(t)} dt < \infty.
	\end{equation}
	Fix $T=t_0+1$. For a nonzero fixed tangent vector $v$, the flow equation gives
	\begin{equation*}
		\left| \partial_t \log g(t)(v,v) \right| \le 2 \|J\|_{\infty,g(t)}.
	\end{equation*}
	Consequently there is a time-independent constant $\Lambda<\infty$ such that
	\begin{equation}\label{eq:J-uniform-metric-equivalence}
		\Lambda^{-1} g(T) \le g(t) \le \Lambda g(T), \qquad t \ge T.
	\end{equation}

	We next control derivatives relative to the fixed connection $\nabla^{g(T)}$. The connection variation formula reads
	\begin{equation*}
		\partial_t \left(\nabla^{g(t)} - \nabla^{g(T)} \right)^k_{ij} = -g^{kl} \left(\nabla_i J_{jl} + \nabla_j J_{il} - \nabla_l J_{ij}\right).
	\end{equation*}
	By \Cref{eq:J-uniform-metric-equivalence,eq:J-integrable-derivatives}, its right-hand side is integrable in the fixed supremum norm. The connection difference, initially zero at $T$, is therefore uniformly bounded. For higher orders, suppose its fixed derivatives of orders less than $q$ are uniformly bounded. Expanding $(\nabla^{g(T)})^q(g^{-1} * \nabla J)$ in terms of $\nabla^{g(t)}$ introduces only these lower derivatives of the connection difference and covariant derivatives $\nabla^\ell J$ with $1 \le \ell \le q+1$. Hence
	\begin{equation*}
		\left\|\partial_t(\nabla^{g(T)})^q \left( \nabla^{g(t)} - \nabla^{g(T)} \right) \right\|_{\infty,g(T)} \le C_q\sum_{\ell=1}^{q+1} \|\nabla^\ell J\|_{\infty,g(t)}.
	\end{equation*}
	The right-hand side is integrable, so induction bounds every fixed derivative of the connection difference. Converting derivatives of $J$ between the two connections now gives
	\begin{equation*}
		\|(\nabla^{g(T)})^m J \|_{\infty, g(T)}
		\le C_m \sum_{\ell=0}^m \|\nabla^\ell J\|_{\infty, g(t)}.
	\end{equation*}
	Thus every fixed $C^m$ norm of $J$ is integrable in time. For $t_2 \ge t_1 \ge T$,
	\begin{equation*}
		\|g(t_2) - g(t_1)\|_{C^m(g(T))}
		\le 2 \int_{t_1}^{t_2} \|J_{g(t)}\|_{C^m(g(T))} dt
		\longrightarrow 0
	\end{equation*}
	as $t_1 \to \infty$. It follows that the entire family converges smoothly to a metric $g_\infty$, which is positive definite by \Cref{eq:J-uniform-metric-equivalence}. Smooth convergence also implies $\|J_{g(t)}\|_{2,g(t)} \to \|J_{g_\infty}\|_{2,g_\infty}$. A positive limit contradicts \Cref{eq:J-time-integrability}, so $J_{g_\infty} = 0$. Norms defined by any other fixed smooth background metric are equivalent to those defined by $g(T)$.
\end{proof}

\section{Variational structure and conditional convergence}\label{sec:variational}

We compute the first variation and evolution of the total $Q$-curvature in general dimension. In dimension four, monotonicity yields a quantitative Yamabe lower bound under positive initial Yamabe constant and positive total $Q$-curvature. A conformal perturbation calculation shows that the same energy can decrease in higher dimensions. We then combine the four-dimensional identities with \Cref{theorem:D,theorem:E} to prove \Cref{theorem:F}.

\subsection{Variational structure and energy monotonicity}\label{subsec:variational}

To study the long-time behavior of the flow, we turn to the total $Q$-curvature
\begin{equation*}
	\mathcal{J}(g) := \int_M Q_g dv_g.
\end{equation*}
We compute its first variation and evolution in general dimension, then use the four-dimensional identities to obtain monotonicity and topological constraints.

\begin{proposition}\label{prop:first-variation}
	Let $h \in S_2(M)$ be a symmetric $2$-tensor. Then
	\begin{equation}\label{eq:first-variation-1}
		D \mathcal{J}_g \cdot h = \int_M \Gamma_g h dv_g + \frac{1}{2} \int_M Q_g \tr_g h dv_g ,
	\end{equation}
	and, using the formal adjoint identity and $J_g = -\frac{1}{2} \Gamma_g^*1$, this can be rewritten as
	\begin{equation}\label{eq:first-variation-2}
		D \mathcal{J}_g \cdot h = \int_M \langle \Gamma_g^*1 + \frac{1}{2} Q_g g, h \rangle_g dv_g = \int_M \langle - 2 J_g + \frac{1}{2} Q_g g, h \rangle_g dv_g .
	\end{equation}
	Consequently the $L^2$-gradient of $\mathcal{J}$ with respect to the standard $L^2$-metric on symmetric $2$-tensors is
	\begin{equation}\label{eq:gradient-J}
		\mathrm{grad} \mathcal{J} (g) = -2 J_g + \frac{1}{2} Q_g g .
	\end{equation}
\end{proposition}

\begin{proof}
	Let $h \in S_2(M)$. The total $Q$-curvature is $\mathcal J(g) = \int_M Q_g dv_g$. Differentiating under the integral sign gives
	\begin{equation*}
		D \mathcal J_g \cdot h = \int_M (D Q_g \cdot h) dv_g + \int_M Q_g (D dv_g \cdot h).
	\end{equation*}
	By definition of the linearization, $D Q_g \cdot h = \Gamma_g h$. The standard variational formula for the volume element is
	\begin{equation*}
		D dv_g \cdot h = \frac{1}{2} (\tr_gh) dv_g.
	\end{equation*}
	Substituting these two identities yields \Cref{eq:first-variation-1}.

	To pass to \Cref{eq:first-variation-2}, recall that the $L^2$-inner product on symmetric $2$-tensors is $\int_M \langle u,v \rangle_g dv_g$, and that the formal adjoint satisfies
	\begin{equation*}
		\int_M (\Gamma_g h) dv_g = \int_M \langle \Gamma_g^*1, h \rangle_g dv_g .
	\end{equation*}
	Therefore
	\begin{equation*}
		D \mathcal J_g \cdot h = \int_M \left\langle \Gamma_g^*1 + \frac{1}{2} Q_g g, h \right\rangle_g dv_g .
	\end{equation*}
	Using $J_g = -\frac{1}{2} \Gamma_g^*1$ we obtain \Cref{eq:first-variation-2} and hence \Cref{eq:gradient-J}.
\end{proof}

The following computation uses only the defining relation $J_g = -\frac{1}{2} \Gamma_g^*1$ and the volume evolution \Cref{eq:volume-evolution}.

\begin{proposition}\label{thm:Q-evolution}
	Let $g(t)$ be a solution of the $J$-flow $\frac{\partial}{\partial t} g = -2 J_g$ on $M \times [0,T)$. Then
	\begin{equation}\label{eq:Q-evolution}
		\frac{d}{dt} \mathcal{J} (g(t)) = 4 \int_M |J_{g(t)}|_{g(t)}^2 dv_{g(t)} - \int_M Q_{g(t)}^2 dv_{g(t)} .
	\end{equation}
	Equivalently,
	\begin{equation}\label{eq:Q-evolution-rewritten}
		\frac{d}{dt} \mathcal{J} (g(t)) = 4 \int_M \left|J_{g(t)} - \frac{1}{n} Q_{g(t)}g(t) \right|_{g(t)}^2 dv_{g(t)} - \frac{n-4}{n} \int_M Q_{g(t)}^2 dv_{g(t)} .
	\end{equation}
	In particular, when $n = 4$, the total $Q$-curvature is non-decreasing along the $J$-flow.
\end{proposition}

\begin{proof}
	Since $\mathcal{J}(g) = \int_M Q dv_g$, the product rule gives
	\begin{equation*}
		\frac{d}{dt} \mathcal{J} (g) = \int_M \frac{\partial Q}{\partial t} dv_g + \int_M Q \frac{\partial}{\partial t} dv_g .
	\end{equation*}
	The first term is computed by linearizing the $Q$-curvature along the flow:
	\begin{equation*}
		\frac{\partial Q}{\partial t} = D Q_g \cdot \left(\frac{\partial g}{\partial t} \right) = \Gamma_g( -2J) .
	\end{equation*}
	Using the defining relation $J = -\frac{1}{2} \Gamma_g^*1$, we obtain
	\begin{equation*}
		\int_M \Gamma_g (-2J) dv_g = \int_M \langle \Gamma_g^*1, -2J \rangle_g dv_g   = \int_M \langle - 2J, -2J \rangle_g dv_g  = 4 \int_M |J|_g^2 dv_g .
	\end{equation*}
	For the volume term we use the standard variational formula
	\begin{equation*}
		\frac{\partial}{\partial t} dv_g = \frac{1}{2} \tr_g \left(\frac{\partial g}{\partial t} \right) dv_g = \frac{1}{2} \tr_g( -2J) dv_g = -Q dv_g .
	\end{equation*}
	Consequently,
	\begin{equation*}
		\int_M Q \frac{\partial}{\partial t} dv_g = -\int_M Q^2 dv_g .
	\end{equation*}
	Adding the two contributions gives \Cref{eq:Q-evolution}. To obtain \Cref{eq:Q-evolution-rewritten}, we decompose $J$ into its trace and traceless parts. Because $\tr_g J = Q$,
	\begin{equation*}
		J = \frac{1}{n} Q g + \left(J - \frac{1}{n} Q g \right), \qquad \left|J - \frac{1}{n} Q g \right|_g^2 = |J|_g^2 - \frac{1}{n} Q^2 .
	\end{equation*}
	Therefore
	\begin{equation*}
		4|J|_g^2 - Q^2 = 4 \left|J - \frac{1}{n} Q g \right|_g^2 + \frac{4}{n} Q^2 - Q^2 = 4 \left|J - \frac{1}{n} Q g \right|_g^2 - \frac{n-4}{n} Q^2 .
	\end{equation*}
	Integrating over $M$ yields \Cref{eq:Q-evolution-rewritten}. When $n = 4$ the last term vanishes, so $\frac{d}{dt} \mathcal J(g) \ge 0$.
\end{proof}

\begin{corollary}\label{cor:Q-energy-identity}
	For any solution of the $J$-flow on $[0,T)$ and any $0 \le t_1 \le t_2 < T$, we have
	\begin{equation*}
		\mathcal{J} (g(t_2)) - \mathcal{J} (g(t_1)) = 4 \int_{t_1}^{t_2} \int_M |J_{g(t)}|_{g(t)}^2 dv_{g(t)} dt - \int_{t_1}^{t_2} \int_M Q_{g(t)}^2 dv_{g(t)} dt .
	\end{equation*}
\end{corollary}

\begin{remark}\label{rmk:traceless-gradient}
	Decomposing a variation $h$ into its trace part $\frac{1}{n}(\tr_g h)g$ and its traceless part $\mathring h$, the traceless contribution to $D \mathcal{J}_g \cdot h$ is
	\begin{equation*}
		-2 \int_M \langle \mathring J_g,\mathring h \rangle_g dv_g, \qquad \mathring J_g = J_g - \frac{1}{n} Q_g g .
	\end{equation*}
	In dimension four, $\mathring J_g = -\frac{1}{2} B_g$, so the trace-free part of $\partial_t g$ is $(\partial_t g)^\circ = B_g$. Thus the Bach tensor governs the transverse-to-conformal variation of $\mathcal{J}$.
\end{remark}

In dimension four, the Chern--Gauss--Bonnet formula reads
\begin{equation}\label{eq:CGB-4}
	8 \pi^2 \chi(M^4) = \int_{M^4} \left( Q_g + \frac{1}{4} |W_g|_g^2 \right) dv_g .
\end{equation}
Since the Euler characteristic is a topological invariant, this yields a uniform bound on the total $Q$-curvature in terms of the Weyl energy.

Together with the curvature decomposition, this formula gives a lower bound for the full curvature energy of every metric, without assuming Bach-flatness or constant $Q$.

\begin{proposition}\label{prop:gap}
	Every metric on a closed four-manifold satisfies
	\begin{equation}\label{eq:topological-energy-identity}
		\int_M |\Rm|^2 dv_g = 32 \pi^2 \chi(M) + 4 \int_M |E|^2 dv_g.
	\end{equation}
	If $M$ is simply connected, then
	\begin{equation*}
		\int_M |\Rm|^2 dv_g \ge 64 \pi^2.
	\end{equation*}
	Equality holds precisely when $E = 0$ and $b_2(M) = 0$. In particular this lower bound applies to both $J$-Einstein and $J$-flat metrics.
\end{proposition}

\begin{proof}
	In dimension four,
	\begin{equation*}
		\mathcal J(g) = -\frac{1}{2} \int_M |E|^2 dv_g + \frac{1}{24} \int_M R^2 dv_g.
	\end{equation*}
	Multiplying \Cref{eq:CGB-4} by four and subtracting the result from the decomposition of $|\Rm|^2$ gives \Cref{eq:topological-energy-identity}. A simply connected closed four-manifold is orientable and has $\chi(M) = 2+b_2(M)$. Both remaining terms are nonnegative, proving the bound and its equality statement.
\end{proof}

\begin{proposition}\label{prop:conformal-invariance}
	The total $Q$-curvature $\mathcal{J}(g)$ is conformally invariant in dimension four, if $\tilde g = e^{2u} g$, then $\mathcal{J}(\tilde g) = \mathcal{J}(g)$.
\end{proposition}

\begin{proof}
	Under a conformal change $\tilde g = e^{2u}g$ the $Q$-curvature transforms according to
	\begin{equation*}
		Q_{\tilde g} = \left(P_g u + Q_g \right) e^{-4u}, \qquad dv_{\tilde g} = e^{4u} dv_g,
	\end{equation*}
	where $P_g$ is the Paneitz operator, which is formally self-adjoint and satisfies $P_g 1 = 0$. Therefore
	\begin{equation*}
		\mathcal{J} (\tilde g) = \int_M (P_gu + Q_g) dv_g = \int_M u P_g1 dv_g + \mathcal{J} (g) = \mathcal{J} (g).
	\end{equation*}
\end{proof}

\begin{corollary}\label{cor:Weyl-monotone}
	Let $g(t)$ be a solution of the $J$-flow on a closed four-manifold. Then the Weyl energy
	\begin{equation*}
		\mathcal{W} (g) := \int_{M^4} |W_g|_g^2 dv_g
	\end{equation*}
	is non-increasing in time.
\end{corollary}

\subsection{Yamabe constant along the flow}
The four-dimensional monotonicity also controls the Yamabe constant when the initial total $Q$-curvature is positive. Recall the definition of Yamabe constant
\begin{equation}\label{eq:Yamabe-constant}
	Y(M,[g]) = \inf_{u \in C^\infty(M) \setminus \{0\}} \frac{\int_M \left( 6 |\nabla u|_g^2 + R_g u^2 \right) dv_g}{\left( \int_M |u|^4 dv_g \right)^{\frac{1}{2}}}.
\end{equation}

\begin{proposition}\label{prop:Yamabe-propagation}
	Let $(M^4, g_0)$ be a closed connected Riemannian four-manifold, and let $g(t)$ be a smooth solution of the $J$-flow on $M^4 \times [0,T)$ with $g(0) = g_0$. If $Y(M, [g_0])> 0$ and $\mathcal J(g_0) > 0$, then
	\begin{equation}\label{eq:Yamabe-Q-lower-bound}
		Y(M,[g(t)]) \ge \sqrt{24 \mathcal J(g(t))} \ge \sqrt{24 \mathcal J(g_0)} > 0
	\end{equation}
	for $0 \le t < T$. No curvature bound or uniform Sobolev inequality along the solution is required for this conclusion.
\end{proposition}
\begin{proof}
	By \Cref{thm:Q-evolution}, $\mathcal J(g(t)) \ge \mathcal J(g_0) > 0$ throughout the smooth existence interval. We first show that the first eigenvalue of the conformal Laplacian $-6 \Delta_{g(t)} + R_{g(t)}$ remains positive. At $t = 0$, the Yamabe inequality and H\"older's inequality give
	\begin{equation*}
		\lambda_1( -6 \Delta_{g_0} + R_{g_0}) \ge \frac{ Y(M,[g_0]) }{ \Vol_{g_0}(M)^{\frac{1}{2}}} > 0.
	\end{equation*}
	The first eigenvalue is continuous on every compact time interval contained in $[0,T)$. If positivity failed, there would be a first time $t_*<T$ at which this eigenvalue is zero. Its first eigenfunction $v$ is smooth and strictly positive, since $M$ is connected. Therefore, we have 
	\begin{equation*}
		(-6 \Delta_{g(t_*)} + R_{g(t_*)}) v = 0,
	\end{equation*}
	and using $v$ as a conformal factor, we obtain
	\begin{equation*}
		R_{v^2 g(t_*)} = v^{-3}(-6 \Delta_{g(t_*)} + R_{g(t_*)}) v = 0.
	\end{equation*}
	Conformal invariance from \Cref{prop:conformal-invariance} and the four-dimensional formula for $Q$ then imply
	\begin{equation*}
		\mathcal J(g(t_*)) = \mathcal J(v^2 g (t_*)) = -\frac{1}{2} \int_M |\Ric_{v^2 g (t_*)}|^2 dv_{v^2 g(t_*)} \le 0,
	\end{equation*}
	contradicting $\mathcal J(g(t_*)) > 0$. Hence the first eigenvalue remains positive. At each time its positive first eigenfunction produces a conformal representative with everywhere positive scalar curvature. On that fixed closed representative, the positive minimum of the scalar curvature and the ordinary Sobolev inequality make the Yamabe quotient positive. By conformal invariance of the Yamabe constant, $Y(M,[g(t)]) > 0$. This argument is applied at each time separately and assumes no uniform control of the auxiliary representatives.

	Fix $t < T$. The Yamabe theorem~\cite{Schoen1984} supplies a unit-volume metric $\widehat g$ in $[g(t)]$ with constant scalar curvature $R_{\widehat g} = Y(M,[g(t)])$. No continuous choice of these metrics is needed. Integrating the Laplacian term in $Q$ gives
	\begin{equation*}
		\begin{split}
			\mathcal J(g(t)) 
			= \mathcal J(\widehat g) 
			= \frac{1}{24} \int_M R_{\widehat g}^2 dv_{\widehat g} - \frac{1}{2} \int_M \left|\Ric_{\widehat g} - \frac{1}{4} R_{\widehat g}\widehat g \right|^2 dv_{\widehat g}
			\le \frac{1}{24} Y(M,[g(t)])^2.
		\end{split}
	\end{equation*}
	Taking square roots and using the monotonicity of $\mathcal J$ proves \Cref{eq:Yamabe-Q-lower-bound}.
\end{proof}

\begin{remark}\label{rmk:Yamabe-Sobolev-scope}
	The lower bound in \Cref{eq:Yamabe-Q-lower-bound} increases with $\mathcal J$, but it does not establish monotonicity of the Yamabe constant itself. It gives, for every smooth $u$,
	\begin{equation*}
		\|u\|_{4,g(t)}^2 \le \frac{1}{\sqrt{24 \mathcal J(g_0)}} \left( 6 \|\nabla u \|_{2,g(t)}^2 + \int_M (R_{g(t)})_+ u^2 dv_{g(t)}\right),
	\end{equation*}
	where $(R_g)_+=\max\{R_g,0\}$. On the original metrics, this estimate alone does not bound the last integral by a time-independent multiple of $\|u\|_2^2$, a uniform upper bound for $(R_{g(t)})_+$ would suffice. The rescaled metrics used in the proof of \Cref{theorem:D} admit such a bound, as explained in \Cref{rmk:Yamabe-curvature-scale}. Moreover, positivity of $\mathcal J(g_0)$ forces the maximal original flow to have finite lifetime, since
	\begin{equation*}
		\Vol_{g(t)}(M) \le \Vol_{g_0}(M)-t\mathcal J(g_0), \qquad T_{\max} \le \frac{\Vol_{g_0}(M)}{\mathcal J(g_0)}.
	\end{equation*}
	Thus \Cref{prop:Yamabe-propagation} controls the conformal classes up to the smooth existence endpoint, it does not assert long-time existence.
\end{remark}

\begin{remark}\label{rmk:Yamabe-curvature-scale}
	Suppose that a four-dimensional solution satisfies $Y(M, [g(t)]) \ge Y_*>0$ throughout its maximal existence interval. This lower bound supplies the Sobolev inequality needed in the rescaling argument of \Cref{theorem:D}. Indeed, the rescaled solutions $\hat g_i$ in that proof satisfy $|\Rm_{\hat g_i(s)}|\le 1$ for $0 \le s \le1$. Scale invariance of the Yamabe constant and \Cref{eq:Rm-norm} give
	\begin{equation*}
		Y(M,[\hat g_i(1)])=Y(M,[g(t_i)]) \ge Y_*, \qquad |R_{\hat g_i(1)}| \le \sqrt{6}.
	\end{equation*}
	Consequently, the defining Yamabe inequality implies
	\begin{equation*}
		\begin{split}
			\|u\|_{4, \hat g_i(1)}^2 \le \frac{1}{Y_*} \left(6 \|\nabla u\|_{2, \hat g_i(1)}^2+ \sqrt{6} \|u\|_{2, \hat g_i(1)}^2 \right) \le \frac{6}{Y_*} \left( \|\nabla u\|_{2, \hat g_i(1)}^2 + \|u\|_{2, \hat g_i(1)}^2 \right).
		\end{split}
	\end{equation*}
	Thus \Cref{lem:integral-pointwise} applies with $C_S = \frac{6}{Y_*}$, and the proof of \Cref{theorem:D} proceeds with the corresponding small-energy threshold, provided the curvature-energy smallness assumption is retained throughout the maximal interval. The curvature bound used here follows from the choice of the first maximum times and requires no prior global curvature estimate. In \Cref{theorem:E}, the assumed bound $|\Rm| \le K$ gives directly
	\begin{equation*}
		\|u\|_{4,g(t)}^2 \le \frac{1}{Y_*} \left(6 \|\nabla u \|_{2,g(t)}^2 + \sqrt{6} K \|u\|_{2,g(t)}^2 \right),
	\end{equation*}
	so a uniform positive Yamabe lower bound also supplies its Sobolev hypothesis, for example with $r_0 = 1$ and $C_S = \frac{6 + \sqrt{6} K}{Y_*}$.

	There is, however, a restriction on the possible energy thresholds. Taking $u=1$ in the Yamabe quotient and using \Cref{eq:Rm-norm} yields
	\begin{equation*}
		Y_* \le Y(M,[g(t)]) \le \frac{\int_M R_{g(t)} dv_{g(t)}}{\Vol_{g(t)}(M)^{\frac{1}{2}}} \le \| R_{g(t)} \|_{2,g(t)} \le \sqrt{6} \|\Rm_{g(t)}\|_{2,g(t)}, 
	\end{equation*}
	which implies
	\begin{equation*}
		\int_M |\Rm_{g(t)}|^2dv_{g(t)} \ge \frac{Y_*^2}{6}.
	\end{equation*}
	A strict small-energy assumption with threshold at most $\frac{Y_*^2}{6}$ would therefore be empty. The substitution above explains the analytic role of the Yamabe lower bound, but does not by itself give a stronger small-energy theorem. The Sobolev formulations in \Cref{theorem:D,theorem:E} also allow Yamabe constant zero, including flat metrics on tori, and are retained for this reason.
\end{remark}

The monotonicity of total $Q$-curvature does not extend to dimensions greater than four, even arbitrarily close to a flat metric. The following calculation uses only the curvature formulas established above.

\begin{proposition}\label{prop:total-Q-high-dimensional-decrease}
	Let $(M^n,\bar g)$ be a closed connected flat manifold with $n > 4$, and let $\varphi$ be a nonzero smooth function satisfying $\int_M \varphi dv_{\bar g} = 0$. Set
	\begin{equation*}
		g_0^\varepsilon = \left( \frac{\Vol_{\bar g}(M)}{ \Vol_{e^{2 \varepsilon \varphi}\bar g}(M)} \right)^{\frac{2}{n}} e^{2 \varepsilon \varphi}\bar g,
	\end{equation*}
	and let $g^\varepsilon(t)$ be the smooth $J$-flow with initial metric $g_0^\varepsilon$, provided by \Cref{theorem:A}. Then
	\begin{equation}\label{eq:total-Q-high-dimensional-decrease}
		\left. \frac{d}{dt} \mathcal J(g^\varepsilon(t)) \right|_{t=0} = -\frac{3(n-4)}{4(n-1)} \varepsilon^2 \int_M ( \Delta_{\bar g}^2 \varphi )^2 dv_{\bar g} + O(\varepsilon^3).
	\end{equation}
	In particular, the derivative is strictly negative for every sufficiently small nonzero $\varepsilon$.
\end{proposition}
\begin{proof}
	All derivatives and norms in this proof are taken with respect to $\bar g$. $\int_M \varphi dv_{\bar g} = 0$ implies that the volume has zero first derivative at $\varepsilon = 0$. Consequently $\left.\partial_\varepsilon g_0^\varepsilon \right|_{\varepsilon = 0} = 2 \varphi \bar g$. The conformal variations of Ricci and scalar curvature at the flat background are
	\begin{equation*}
		D\Ric_{\bar g}[2 \varphi \bar g] = -(n-2) \nabla^2 \varphi - (\Delta \varphi) \bar g, \qquad DR_{\bar g}[2 \varphi\bar g] = -2(n-1) \Delta \varphi.
	\end{equation*}
	The quadratic curvature terms in \Cref{eq:intro-Q,eq:J-full} have zero linearization at $\bar g$, as do variations of the differential operators acting on its vanishing curvature. Substitution therefore gives
	\begin{equation*}
		D Q_{\bar g}[2 \varphi \bar g] = \Delta^2\varphi, \qquad
		D J_{\bar g}[2 \varphi \bar g] = \frac{n-4}{4(n-1)} \nabla^2 \Delta \varphi + \frac{3}{4(n-1)}(\Delta^2 \varphi)\bar g.
	\end{equation*}
	Flatness and integration by parts give
	\begin{equation*}
		\int_M |\nabla^2 \Delta \varphi|^2 dv_{\bar g} = \int_M (\Delta^2 \varphi)^2 dv_{\bar g}.
	\end{equation*}
	Expanding the square thus yields
	\begin{equation*}
		\begin{split}
			\int_M \left| DJ_{\bar g}[2 \varphi \bar g] \right|^2 dv_{\bar g} 
			= \frac{(n-4)^2+6(n-4)+9n}{16(n-1)^2} \int_M (\Delta^2 \varphi)^2 dv_{\bar g} 
			= \frac{n+8}{16(n-1)} \int_M (\Delta^2 \varphi)^2 dv_{\bar g}.
		\end{split}
	\end{equation*}
	Since $J_{\bar g} = Q_{\bar g} = 0$, Taylor expansion of \Cref{eq:Q-evolution} at $g_0^\varepsilon$ gives
	\begin{equation*}
		\begin{split}
			\left.\frac{d}{dt}\mathcal J(g^\varepsilon(t))\right|_{t=0}
			& = \varepsilon^2 \left(4 \int_M \left| D J_{\bar g}[2 \varphi \bar g] \right|^2 dv_{\bar g} - \int_M \left| D Q_{\bar g}[2 \varphi \bar g] \right|^2 dv_{\bar g} \right) + O(\varepsilon^3) \\
			& = \left(\frac{n+8}{4(n-1)} - 1 \right) \varepsilon^2 \int_M( \Delta^2 \varphi)^2 dv_{\bar g} + O(\varepsilon^3).
		\end{split}
	\end{equation*}
	Variations of the contractions and volume density contribute only to the remainder, since the curvature tensors being squared vanish at $\varepsilon = 0$. This proves \Cref{eq:total-Q-high-dimensional-decrease}. Finally, $\Delta^2 \varphi = 0$ would imply $\int_M (\Delta \varphi)^2 dv_{\bar g} = 0$, so $\varphi$ would be constant and hence zero. The leading integral is therefore strictly positive.
\end{proof}

\begin{remark}\label{rmk:normalized-Yamabe-and-Q}
	Both conclusions hold for the volume-normalized equation as well. In dimension $n\ge4$, this equation is
	\begin{equation*}
		\frac{\partial g}{\partial t} = -2 J_g + \frac{2}{n} \bar Q_g g, \qquad \bar Q_g = \frac{\mathcal J(g)}{\Vol_g(M)},
	\end{equation*}
	which agrees with \Cref{eq:normalized-J-flow} when $n = 4$. By \Cref{eq:first-variation-2}, its normalization term adds
	\begin{equation*}
		D\mathcal J_g \left[ \frac{2}{n} \bar Q_g g \right] = \frac{n-4}{n} \Vol_g(M) \bar Q_g^2
	\end{equation*}
	to \Cref{eq:Q-evolution}. This contribution is zero in dimension four, so the proof and conclusion of \Cref{prop:Yamabe-propagation} remain valid. For the initial metrics in \Cref{prop:total-Q-high-dimensional-decrease}, $\mathcal J(g_0^\varepsilon) = O(\varepsilon^2)$ because $Q_{\bar g} = 0$ and $\int_M \Delta^2 \varphi dv_{\bar g} = 0$. The added term is therefore $O(\varepsilon^4)$, leaving the strictly negative second-order term in \Cref{eq:total-Q-high-dimensional-decrease} unchanged.
\end{remark}

\subsection{Conditional convergence}\label{subsec:convergence}

We prove \Cref{theorem:F} by combining the monotonicity formula, the gradient inequality and time integrability of the $L^2$-norm of $Q$ to establish time integrability of the $L^2$-norm of $J$, as required by \Cref{theorem:E}. We also establish a separate result on sequential limits under a compactness assumption in \Cref{prop:compact-sequential-limits}. Recall that $g$ is $J$-flat if $J_g \equiv 0$, in dimension four, \Cref{eq:J-dim4} gives the following characterization.

\begin{proposition}\label{prop:J-flat-char}
	A Riemannian metric $(M^4, g)$ is $J$-flat if and only if it is Bach-flat and has zero $Q$-curvature.
\end{proposition}

\begin{proof}
	Taking the trace and using that the Bach tensor is traceless in four dimensions gives $\tr_g J_g = Q_g$. Then the result is trivial.
\end{proof}

\begin{remark}
	By \Cref{prop:J-Schur}, the definition $J = \Lambda(x) g$ is equivalent to Bach-flatness in dimension four. Requiring $\Lambda$ to be constant imposes the additional condition that $Q$ is constant. For example, on a flat four-torus any metric $g = e^{2u} g_{\mathrm{flat}}$ is Bach-flat, whereas $Q_g = e^{-4u} \Delta_{g_{\mathrm{flat}}}^2u$ is generally nonconstant. Thus the Schur conclusion cannot be applied in dimension four. The characterization of $J$-flatness in \Cref{prop:J-flat-char} does require $Q = 0$.
\end{remark}

\begin{remark}
	The round sphere is not $J$-flat because its $Q$-curvature is a positive constant. Consequently there is no simply connected closed $J$-flat four-manifold with small total curvature: Chern--Gauss--Bonnet gives $\int_M |W|^2 dv = 32 \pi^2 \chi(M) \ge 64 \pi^2$ for any simply connected $J$-flat metric. The stronger bound in \Cref{prop:gap} is a topological statement for all metrics. 
\end{remark}

Along the four-dimensional flow, \Cref{thm:Q-evolution} and \Cref{eq:J-dim4,eq:volume-evolution} give
\begin{equation}\label{eq:Q-mono-4D}
	\frac d{dt} \mathcal J(g(t)) = \|B_{g(t)}\|_2^2, \qquad \frac d{dt} \Vol_{g(t)}(M)= -\mathcal J(g(t)).
\end{equation}
We keep the metric in the $L^2$ norm time dependent.

\begin{lemma}\label{lem:compact-J-limit}
	For an immortal $J$-flow on a closed four-manifold, $\mathcal J(g(t)) \le 0$ for every $t$. If its volume is uniformly bounded above, then $\mathcal J(g(t)) \to 0$.
\end{lemma}

\begin{proof}
	If $\mathcal J(g(t_1)) > 0$ at any time, monotonicity gives $\Vol'(t) \le -\mathcal J(g(t_1))$ for $t \ge t_1$, which makes the volume negative in finite time. Hence $\mathcal J \le 0$. Its increasing limit $L$ exists. If the volume is bounded above and $L < 0$, then $\Vol'(t) = -\mathcal J(g(t)) \ge -L > 0$, contradicting that bound. Therefore $L = 0$ in this case.
\end{proof}

\begin{lemma}\label{lem:Bach-decay}
	Let $g(t)$ be an immortal $J$-flow on a closed four-manifold. Suppose that, for some $t_0 > 0$, $\{g(t):t \ge t_0\}$ has compact closure in the $C^\infty$ topology of positive definite metrics on the fixed manifold $M$. Then $\|B_{g(t)}\|_{C^k( \bar{g} )} \to 0$ for every $k \ge 0$ and any fixed smooth background metric $\bar{g}$.
\end{lemma}

\begin{proof}
	Compact closure gives a uniform upper bound for the volume, uniform bounds on all fixed-coordinate derivatives of $g$ and $g^{-1}$, and uniform equivalence to $\bar{g}$. The flow equation also bounds all fixed-coordinate derivatives of $\partial_tg$. By \Cref{lem:compact-J-limit} and \Cref{eq:Q-mono-4D},
	\begin{equation*}
		\int_{t_0}^\infty \|B_{g(t)}\|_2^2 dt = - \mathcal J(g(t_0))<\infty.
	\end{equation*}
	The Bach tensor is a differential expression of order four in $g$. The uniform bounds on the metric, its derivatives through order eight, and $\partial_tg$ and its derivatives through order four give a uniform bound for $\left| \frac{d}{dt} \|B_{g(t)}\|_2^2 \right|$. Therefore $\|B_{g(t)}\|_2^2 \to 0$, since a nonnegative integrable function with bounded derivative tends to zero, otherwise uniform continuity supplies disjoint intervals of definite length on which its integral is bounded below.

	To obtain every $C^k$ conclusion, suppose there were a sequence $t_i \to \infty$ along which the asserted norm stayed positive. Compactness gives a smoothly convergent subsequence of $g(t_i)$, continuity of $B$ and the vanishing $L^2$ norm make its limiting Bach tensor zero, a contradiction.
\end{proof}

\begin{proposition}\label{prop:compact-sequential-limits}
	Let $(M^4, g_0)$ be a closed Riemannian four-manifold, and let $g(t)$ be a smooth solution of the $J$-flow on $M^4 \times [0,\infty)$ with $g(0) = g_0$. Suppose that, for some $t_0 > 0$, $\{g(t): t \ge t_0\}$ has compact closure in the $C^\infty$ topology of positive definite metrics on the fixed manifold $M$. Then
	\begin{equation*}
		\mathcal J(g(t)) \nearrow 0, \qquad \|B_{g(t)}\|_{C^k(\bar{g})} \longrightarrow 0
	\end{equation*}
	for every $k \ge 0$ and any fixed smooth background metric $\bar{g}$. Every sequence $t_i \to \infty$ has a subsequence along which $g(t_i)$ converges smoothly to a positive definite Bach-flat metric with total $Q$-curvature zero. If $\| Q_{g(t)} \|_{L^2(g(t))}\to 0$, every such limit is $J$-flat.
\end{proposition}

\begin{proof}
	Compact closure bounds the volume above, so \Cref{lem:compact-J-limit,lem:Bach-decay} give $\mathcal J  \nearrow  0$ and $B \to 0$. Compact closure provides a smoothly convergent subsequence from every sequence of times tending to infinity. Each limit is Bach-flat with total $Q$-curvature zero. If $\|Q\|_2 \to 0$, continuity under smooth metric convergence gives $Q = 0$ at each limit, and \Cref{prop:J-flat-char} makes it $J$-flat.
\end{proof}

To prove convergence of the entire trajectory, we combine the gradient-inequality hypothesis with the monotonicity formula to show that the $L^2$-norm of the Bach tensor is integrable in time.

\begin{lemma}\label{lem:Bach-time-integrability}
	Let $g(t)$ be an immortal $J$-flow on a closed four-manifold satisfying \Cref{eq:LS-hypothesis-F} for all sufficiently large $t$. Then $\mathcal J(g(t)) \nearrow 0$, and for all sufficiently large $a$,
	\begin{equation}\label{eq:Bach-tail-estimate}
		\int_a^\infty \|B_{g(t)}\|_2 dt \le \frac{C_{\mathrm L}}{\theta} \left(-\mathcal J(g(a)) \right)^\theta.
	\end{equation}
\end{lemma}

\begin{proof}
	By \Cref{lem:compact-J-limit} and \Cref{eq:Q-mono-4D}, $\mathcal J(g(t)) \le 0$ and
	\begin{equation*}
		\frac{d}{dt} \mathcal J(g(t)) = \| B_{g(t)} \|_2^2 \ge 0.
	\end{equation*}
	Its increasing limit is zero: a negative limit and \Cref{eq:LS-hypothesis-F} would give a positive lower bound for $\| B_{g(t)} \|_2$, forcing $\mathcal J(g(t))$ to become positive. On intervals where $\mathcal J(g(t)) < 0$ and \Cref{eq:LS-hypothesis-F} holds, we have
	\begin{equation*}
		-\frac{d}{dt} \left(-\mathcal J(g(t)) \right)^\theta = \theta \left(-\mathcal J(g(t))\right)^{\theta-1} \|B_{g(t)}\|_2^2 \ge \frac{\theta}{C_{\mathrm L}} \|B_{g(t)}\|_2.
	\end{equation*}
	Integrating from $a$ to infinity and using $\mathcal J(g(t)) \to 0$ proves the estimate when $\mathcal J(g(t))$ stays negative. If it reaches zero at a finite time, it remains zero by its nonpositivity and monotonicity, and then $\| B_{g(t)} \|_2 = 0$ thereafter. Integrating up to that time proves the estimate in this case as well.
\end{proof}

\begin{proof}[Proof of \Cref{theorem:F}]
	The proof of \Cref{theorem:D} gives a uniform curvature bound on $[0, \infty)$, and the uniform Sobolev inequality is part of its hypotheses. It remains to verify that the $L^2$-norm of $J$ is integrable in time, as required by \Cref{theorem:E}. Since $B$ is trace free and $|g|_g^2 = 4$, \Cref{eq:J-dim4} gives
	\begin{equation}\label{eq:velocity-orthogonal-F}
		4\|J\|_{L^2(g(t))}^2 = \|\partial_t g\|_{L^2(g(t))}^2 = \|B\|_{L^2(g(t))}^2 + \|Q\|_{L^2(g(t))}^2.
	\end{equation}
	For all sufficiently large $a$, \Cref{lem:Bach-time-integrability,eq:Q-time-integrability-F} therefore imply
	\begin{equation}\label{eq:J-tail-estimate-F}
		\int_a^\infty\|J\|_2 dt \le \frac{C_{\mathrm L}}{2\theta} \left(-\mathcal J(g(a)) \right)^\theta + \frac{1}{2} \int_a^\infty \|Q\|_2 dt < \infty.
	\end{equation}
	Since the solution is smooth on the compact cylinder $M \times [0, a]$, the integral of $\|J\|_2$ over $[0, a]$ is finite. Combining this with \Cref{eq:J-tail-estimate-F} gives
	\begin{equation*}
		\int_0^\infty \|J_{g(t)}\|_{L^2(g(t))} dt < \infty.
	\end{equation*}
	In particular, \Cref{eq:J-time-integrability} holds, and \Cref{theorem:E} gives smooth convergence in fixed coordinates to a positive definite $J$-flat metric.
\end{proof}

\begin{corollary}\label{cor:Bach-Q-integrability}
	In \Cref{theorem:F}, one may replace the gradient inequality by
	\begin{equation*}
		\int_0^\infty \|B_{g(t)}\|_2 dt < \infty.
	\end{equation*}
	The hypotheses of \Cref{theorem:D} and the time-integrability condition on the $L^2$-norm of $Q$ in \Cref{eq:Q-time-integrability-F} are retained.
\end{corollary}

\begin{proof}
	By \Cref{eq:velocity-orthogonal-F}, these assumptions give
	\begin{equation*}
		\int_0^\infty \|J\|_2 dt \le \frac{1}{2} \int_0^\infty \left( \|B\|_2 + \|Q\|_2 \right) dt < \infty.
	\end{equation*}
	Thus \Cref{eq:J-time-integrability} holds. The curvature bound from the proof of \Cref{theorem:D} and its Sobolev hypothesis allow us to apply \Cref{theorem:E}.
\end{proof}

\section{J soliton and future work}\label{sec:J-soliton-future}

Ricci solitons are the self-similar solutions of the Ricci flow: their evolution consists of scaling and pullback by diffeomorphisms. Their defining equation and its gradient form are
\begin{equation*}
    \Ric_g + \frac{1}{2} \mathcal L_X g = \lambda g, \qquad \Ric_g + \nabla^2 f = \lambda g,
\end{equation*}
where $X$ is a smooth vector field and $\lambda$ is a constant. Perelman's entropy functionals relate these equations to variational structures and monotonicity along the Ricci flow~\cite{Perelman2002}. On surface, thhe Ricci soliton has to be Einstein~\cite{Cao2010survey}. 

For the $J$-flow, we can define the corresponding self-similar equation
\begin{equation}\label{eq:J-soliton}
    J_g + \frac{1}{2} \mathcal L_X g = \lambda g,
\end{equation}
where $X$ is a smooth vector field and $\lambda$ is a constant. The soliton is shrinking, steady or expanding according as $\lambda > 0$, $\lambda = 0$ or $\lambda < 0$, and is called gradient if $X = \nabla f$. We call it trivial when $J_g = \lambda g$ and $X$ is Killing. In dimension four, triviality is forced by the trace and divergence identities together with the Kazdan--Warner identity for $Q$-curvature.

\begin{proposition}\label{prop:compact-J-soliton}
    Let $(M^4,g)$ be a closed connected Riemannian four-manifold satisfying \Cref{eq:J-soliton}, then every such soliton is trivial with constant Q curvature and Bach-flat. Furthermore, if $X = \nabla f$, then $f$ is constant. Conversely, every Bach-flat metric with constant $Q$-curvature satisfies \Cref{eq:J-soliton} with $\lambda = \frac{1}{4} Q_g$ and any Killing field $X$.
\end{proposition}

\begin{proof}
    By \Cref{prop:trace-J,prop:J-Schur}, the tensor $J_g - \frac{1}{4} Q_g g$ is trace free and satisfies
    \begin{equation*}
        \divv_g \left( J_g - \frac{1}{4} Q_g g \right) = \frac{1}{4} d Q_g - \frac{1}{4}d Q_g = 0.
    \end{equation*}
    Pairing \Cref{eq:J-soliton} with this tensor and integrating by parts gives
    \begin{equation*}
        \begin{split}
            \int_M \left\langle J_g - \frac{1}{4} Q_g g, J_g - \frac{1}{4} Q_g g \right\rangle dv_g
            & = \int_M \left\langle J_g - \frac{1}{4} Q_g g, \mathcal \lambda g - \frac{1}{2} L_X g - \frac{1}{4} Q_g g \right\rangle dv_g\\
			& = \int_M \left\langle J_g - \frac{1}{4} Q_g g, \mathcal  - \frac{1}{2} L_X g  \right\rangle dv_g\\
            & = \int_M \divv_g \left( J_g - \frac{1}{4} Q_g g \right)(X) d v_g = 0.
        \end{split}
    \end{equation*}
    Thus $J_g = \frac{1}{4} Q_g g$, and $B_g = 0$. Substitution into the soliton equation yields
    \begin{equation*}
        \mathcal L_X g = 2 \left( \lambda - \frac{1}{4} Q_g \right) g,
    \end{equation*}
    so $X$ is a conformal vector field. The Kazdan--Warner identity for $Q$-curvature on a closed manifold~\cite[Proposition~4.2]{GoverOrsted2013} gives
    \begin{equation*}
        \int_M X(Q_g) dv_g = 0.
    \end{equation*}
    On the other hand, taking trace of \Cref{eq:J-soliton} and integrating gives
    \begin{equation*}
        Q_g + \divv_g X = 4 \lambda.
    \end{equation*}
    Consequently,
	\begin{equation*}
		\int_M (Q_g - 4 \lambda)^2 dv_g = \int_M Q_g (Q_g - 4 \lambda) dv_g = - \int_M Q_g \divv_g X dv_g = \int_M X(Q_g) dv_g = 0.
	\end{equation*}
    Hence $Q_g = 4 \lambda$ and $\mathcal L_X g = 0$. If $X = \nabla f$, then $\nabla^2 f=0$, so $\Delta_g f = 0$ and $f$ is constant by closeness and connectedness. The converse follows from \Cref{eq:J-dim4}.
\end{proof}

This result also shows that every closed four-dimensional $J$-soliton metric admits a gradient representation with constant potential function $f$.

\section*{Declarations}

\subsection*{Conflict of interest}
The authors declare that they have no competing financial or non-financial interests that are directly or indirectly related to the work submitted for publication.

\subsection*{Declaration of generative AI in scientific writing}
The authors declare that no generative artificial intelligence (AI) or AI-assisted technologies were used in the preparation, analysis of this manuscript.

\bibliographystyle{amsplain}
\bibliography{cas-refs}
\end{document}